\documentclass[a4paper, 10 pt, reqno]{article}
\usepackage{amsfonts, amssymb, amsmath, eucal, amsthm}
\usepackage[all]{xy} 

\newtheorem{lemma}{Lemma}
\newtheorem{theorem}[lemma]{Theorem}
\newtheorem*{theorem*}{Theorem}
\newtheorem{proposition}[lemma]{Proposition}
\newtheorem{corollary}[lemma]{Corollary}

\theoremstyle{definition}
\newtheorem{definition}[lemma]{Definition}
\newtheorem{note}[lemma]{Note}
\newtheorem{example}[lemma]{Example}

\newtheorem{question}[lemma]{Question}
\newtheorem*{acknowledgments}{Acknowledgments}

\numberwithin{lemma}{section}

\newcommand{\R}{\mathbb{R}}
\newcommand{\C}{\mathbb{C}}
\newcommand{\Z}{\mathbb{Z}}
\newcommand{\X}{\mathfrak{X}}
\newcommand{\Y}{\mathfrak{Y}}

\newcommand{\ddt}{\partial/\partial t}
\newcommand{\bigddt}{\frac{\partial}{\partial t}}
\newcommand{\Id}{\mathrm{Id}}
\DeclareMathOperator{\cl}{\mathrm{cl}}
\DeclareMathOperator{\supp}{\mathrm{supp}}
\DeclareMathOperator{\interior}{\mathrm{int}}
\DeclareMathOperator{\image}{\mathrm{im}}
\newcommand{\aut}{\mathrm{Aut}}
\newcommand{\tdf}{\frac{3\delta}{4}}
\newcommand{\pt}{\mathrm{pt}}
\newcommand{\ind}{\mathrm{ind}}
\newcommand{\coind}{\mathrm{coind}}
\newcommand{\orb}{\mathrm{orb}}

\newcommand{\AbramovichVistoliGraber}{MR1950940}
\newcommand{\ALR}{AdemLeidaRuan}
\newcommand{\BanyagaHurtubise}{MR2145196}
\newcommand{\BehrendXu}{BehrendXu}
\newcommand{\RuanCrepant}{MR2234886}
\newcommand{\ChenRuan}{MR2104605}
\newcommand{\ChenRuanGromovWitten}{MR1950941}
\newcommand{\Crepant}{Crepant}

\newcommand{\Heinloth}{Heinloth}
\newcommand{\Hirsch}{MR0448362}
\newcommand{\LermanTolman}{MR1401525}
\newcommand{\Lang}{MR0431240}
\newcommand{\Milnor}{MR0163331}
\newcommand{\JoyceDesingularizations}{JoyceDesingularizations}
\newcommand{\JoyceExceptional}{MR1787733}
\newcommand{\Moerdijk}{MR1950948}
\newcommand{\Noohi}{Noohi}
\newcommand{\OrbifoldMSW}{OrbifoldMSW}
\newcommand{\Invitation}{MR2298610}
\newcommand{\LUXLoop}{MR2340450}
\newcommand{\GinotEtAl}{InertiaString}

\newcommand{\Warner}{MR0295244}
\newcommand{\Wasserman}{MR0250324}

\title{Morse Inequalities for Orbifold Cohomology.}
\author{Richard Hepworth\footnote{The author is supported by E.P.S.R.C.~Postdoctoral Research Fellowship EP/D066980.}\\  Department of Pure Mathematics\\ University of Sheffield}

\begin{document}
\maketitle

\begin{abstract}
This paper begins the study of Morse theory for orbifolds, or more precisely for differentiable Deligne-Mumford stacks.  The main result is an analogue of the Morse inequalities that relates the orbifold Betti numbers of an almost-complex orbifold to the critical points of a Morse function on the orbifold.  We also show that a generic function on an orbifold is Morse.  In obtaining these results we develop for differentiable Deligne-Mumford stacks those tools of differential geometry and topology --- flows of vector fields, the strong topology --- that are essential to the development of Morse theory on manifolds.
\end{abstract}

\setcounter{tocdepth}{1}
\tableofcontents

\section{Introduction}

In this paper we begin the study of Morse theory for orbifolds, or more precisely for differentiable Deligne-Mumford stacks.  Morse theory studies the homology or cohomology of a manifold by looking at the critical points of an appropriate function on the manifold.  The simplest instance of this is the \emph{Morse inequalities}, which relate the Betti numbers of a compact manifold $M$ to the number of critical points of a Morse function on $M$ \cite{\Milnor}.  In extending Morse theory to orbifolds one must choose how to extend the notion of homology or cohomology from manifolds to orbifolds.  There are several options here but the most interesting is the \emph{orbifold} or \emph{Chen-Ruan} cohomology $H^\ast_\orb(\X)$ defined for almost-complex orbifolds $\X$ \cite{\ChenRuan}.  Ruan's \emph{crepant resolution conjecture} relates $H^\ast_\orb(\X)$ to the cohomology $H^\ast(Y)$ of a crepant resolution $Y\to\X$, so that $H^\ast_\orb(\X)$ serves as a tool for understanding the cohomology of crepant resolutions.  The main result of this paper is a generalization of the Morse inequalities  to orbifolds.  It relates the critical points of a Morse function on an almost-complex orbifold $\X$ to the ranks of the orbifold cohomology groups $H^i_\orb(\X)$.  In obtaining this result we develop for differentiable Deligne-Mumford stacks those tools of differential geometry and topology --- partitions of unity, Riemannian metrics, flows of vector fields, the strong topology --- that are so essential for the development of Morse theory on manifolds.

The Morse inequalities are often sufficient to compute the Betti numbers of a manifold, but a much more powerful result is provided by the \emph{Morse-Smale-Witten complex} \cite{MR1045282}.  The Morse-Smale-Witten complex of a Morse-Smale function $f$ on a compact Riemannian manifold $M$ is a chain complex defined in terms of the critical points and gradient flow lines of $f$, and its cohomology is naturally isomorphic to the cohomology of $M$.  In a forthcoming paper \cite{\OrbifoldMSW} we will extend the current work to develop a Morse-Smale-Witten approach to orbifold cohomology, and in \cite{\Crepant} we will demonstrate how the methods of Morse-Smale-Witten theory can be used to compute the integer homology of crepant resolutions of orbifolds.  

For a compact manifold $M$, the \emph{Morse inequalities} relate the Betti numbers of $M$ to the critical points of a Morse function $f\colon M\to\R$.  Let us write
\[P_t(M)=\sum\dim H_i(M;\C) t^i,\qquad M_t(f)=\sum_c t^{\ind_c}\]
for the \emph{Poincar\'e polynomial} of $M$ and the \emph{Morse polynomial} of $f$ respectively, where $c$ runs over critical points of $f$ and $\ind_c$ denotes the index.  Then the Morse inequalities state that 
\begin{equation}\label{MorseInequalitiesEquation}M_t(f)=P_t(M)+(1+t)R(t)\end{equation}
for some polynomial $R(t)$ with non-negative integer coefficients.  This often allows the Betti numbers of $M$ to be computed directly from $f$.  For example, if the critical points of $f$ all have even index then $R(t)=0$ and $P_t(M)=M_t(f)$.  One can always find a Morse function on $M$, and indeed Morse functions form a dense open subset of $C^\infty(M)$ when it is equipped with the strong topology.

Our main result generalizes the Morse inequalities \eqref{MorseInequalitiesEquation} to compact differentiable Deligne-Mumford stacks $\X$.  We will define what it means for $f\colon\X\to\R$ to be a  \emph{Morse function}; in this case $f$ has a discrete set of critical points $c$, each with an automorphism group $\aut_c$ and an index $\ind_c$ which is a linear representation of $\aut_c$.  When $\X$ is compact and almost-complex we define
\[P_t^\orb(\X)=\sum\dim H^i_\orb(\X)t^i,\qquad M_t^\orb(f)=\sum_{c,(g)}t^{\dim({\ind_c}^g)+2\iota(g)}\]
to be the \emph{orbifold Poincar\'e polynomial} of $\X$ and the \emph{orbifold Morse polynomial} of $f$ respectively.  The second sum is taken over pairs $c,(g)$ for which $C_{\aut_c}(g)$ preserves orientations of ${\ind_c}^g$, and $\iota(g)$ denotes the degree-shifting number.

\begin{theorem*}[Orbifold Morse Inequalities]
There is a polynomial $R(t)$ with non-negative integer coefficients such that
\begin{equation}M_t^\orb(f)=P_t^\orb(\X)+(1+t)R^\orb(t).\label{OrbifoldMorseInequalitiesEquation}\end{equation}
\end{theorem*}

In fact we will give two other generalizations of the Morse inequalities,  each corresponding to a different notion of homology or cohomology of an orbifold.  Furthermore, we show that Morse functions always exist by defining a topology on the set $C^\infty(\X)$ of morphisms $\X\to\R$ and proving the following theorem.

\begin{theorem*}
Morse functions form a dense open subset of $C^\infty(\X)$.
\end{theorem*}

The paper is organized as follows.  In Section~\ref{UnderlyingSpacesSection} we introduce the underlying space $\bar\X$ of a differentiable stack $\X$.  This is a topological space which reflects certain `topological' or `pointwise' properties of $\X$.  It will turn out that if $\X$ is differentiable Deligne-Mumford then $\bar\X$ is an orbifold in the traditional sense, though $\bar\X$ is not a priori part of the data of the stack $\X$.  We show that subsets of $\bar\X$ correspond to a particular kind of substack of $\X$; this will be essential in \cite{\OrbifoldMSW} for defining stable and unstable manifolds.  The section ends with a discussion of the differences between the notions of underlying space and coarse moduli space.  In Section~\ref{DDMStacksSection} we recall the definition of differentiable Deligne-Mumford stack and discuss its relation with proper \'etale groupoids and orbifolds in the traditional sense.  We then derive some basic but essential properties of differentiable Deligne-Mumford stacks including paracompactness, partitions of unity and the existence of orbifold-charts.  Section~\ref{MorseFunctionSection} introduces Morse functions on differentiable Deligne-Mumford stacks and defines their critical points and the index and co-index of a critical point.  We prove a Morse Lemma which describes the local form of a Morse function.  Finally we show how a Morse function on $\X$ immediately gives a Morse function on the inertia stack $\Lambda\X$ by composing with the evaluation-map $\Lambda\X\to\X$; this is essential to our results on orbifold cohomology.

Sections \ref{GeometrySection} and \ref{StrongTopologySection} are the technical heart of the paper.  In Section~\ref{GeometrySection} we define Riemannian metrics, vector fields, the gradient vector field, and integrals and flows of vector fields on differentiable Deligne-Mumford stacks.  We then prove results on the existence of Riemannian metrics, uniqueness of integrals, and uniqueness and existence of flows, and we demonstrate by example that these results can fail for stacks that are not Deligne-Mumford.  Section~\ref{StrongTopologySection} defines the strong topology on the set $C^\infty(\X)$ of morphisms $\X\to\R$.  This generalizes the strong topology on $C^\infty(M)$ for a manifold $M$.  We discuss the relation with the strong topology on an atlas $X\to\X$ and we prove that $C^\infty(\X)$ is a Baire space.  Finally we show that Morse functions form a dense open subset of $C^\infty(\X)$, so that in particular every differentiable Deligne-Mumford stack admits a Morse function.

In Section~\ref{ApplicationsSection} we use the material developed so far to derive the main results of the paper.  We examine the homotopy type of the underlying space $\bar\X$ in terms of a Morse function $\X\to\R$, and we then use this result to prove Morse inequalities for each of three notions of homology or cohomology of an orbifold.  We also give a characterization of representable stacks using Morse functions.  Finally, in Section~\ref{ExamplesSection} we give some examples of Morse functions and the conclusions one can draw from the Morse inequalities.  We also give an example that shows how Morse theoretic techniques can be used to compute the integral homology of crepant resolutions of orbifolds.

\begin{acknowledgments}
Thanks to David Gepner for many interesting and useful discussions about stacks.  The author is supported by an E.P.S.R.C.~Postdoctoral Research Fellowship, grant number EP/D066980.
\end{acknowledgments}

\section{Underlying Spaces and Coarse Moduli Spaces.}\label{UnderlyingSpacesSection}

This section deals with differentiable and topological stacks. For generalities on differentiable stacks see \cite{\BehrendXu} or \cite{\Heinloth} and for topological stacks see \cite{\Noohi}.  Throughout this section and the ones following we will often treat \emph{representable} differentiable stacks as manifolds without specifying a particular equivalence.  Thus, for example, given a diagram 
\[\xymatrix{
U\times_\X V\ar[r]\ar[d] & U\ar[d]_{}="1"\\
V\ar[r]^{}="2" & \X \ar@{=>}"1";"2"
}\]
with $\X$ a differentiable stack and $U,V\to\X$ submersions, we will treat $U\times_\X V$ as a manifold without further comment, even though it is only equivalent to a manifold.

In this section we will define the \emph{underlying space} $\bar\X$ of a differentiable stack $\X$.  The underlying space $\bar\X$ is a topological space that reflects many properties of the stack $\X$.  It is also the natural home for many concepts that one can generalize from the theory of manifolds.  For example, the support of a function on a manifold $X$ is a closed subset of $X$, and the support of a function on a stack $\X$ is a closed subset of $\bar\X$.

In the case of differentiable Deligne-Mumford stacks the underlying space is of especial importance.  It provides the link between stacks and orbifolds in the traditional sense.  It also has many of the good topological properties of manifolds, and this allows us to generalize many aspects of the theory of manifolds to differentiable Deligne-Mumford stacks.  This will be explained further in Section~\ref{DDMStacksSection}.

Although we will define the underlying space of a differentiable stack, our construction works just as well for topological stacks.  We then recover an existing construction of Noohi, who in \cite[4.3]{\Noohi} defined for each topologlical stack $\X$ a space $\X_\mathrm{mod}$ together with a morphism $\mathrm{mod}\colon\X\to\X_\mathrm{mod}$ that makes $\X_\mathrm{mod}$ a coarse moduli space of $\X$.  If one applies our definition of underlying space to a topological stack $\X$, then $\bar\X$ is naturally isomorphic to $\X_\mathrm{mod}$.  Indeed, in \cite[4.3]{\Noohi} the term `underlying space' was used to informally refer to $\X_\mathrm{mod}$.

In contrast with the topological case, when dealing with differentiable stacks we cannot call $\bar\X$ the coarse moduli space of $\X$.  Indeed, we will see that not all differentiable stacks admit a coarse moduli space, while some differentiable stacks have a coarse moduli space that is not isomorphic to the underlying space.

The section is organized as follows.  In \S\ref{UnderlyingSpaceSubsection} we define the underlying space of a differentiable stack and establish some of its functorial properties.  In \S\ref{SubstackSubsection} we establish a correspondence between subsets of the underlying space and certain substacks of the original stack.  Then in \S\ref{CoarseModuliSubsection} we consider the material of \S\ref{UnderlyingSpaceSubsection} and \S\ref{SubstackSubsection} in the topological setting and compare it with results from \cite{\Noohi}.  We then show how, in the differentiable setting, the notion of coarse moduli space differs wildly from that of the underlying space.

\subsection{The underlying space of a differentiable stack.}\label{UnderlyingSpaceSubsection}
Let $\X$ be a differentiable stack and consider the groupoid $\X(\pt)$ of geometric points $\pt\to\X$ and $2$-morphisms between them.  If $X\to\X$ is an atlas, then $\X(\pt)$ is equivalent to the groupoid of sets $X\leftleftarrows X\times_\X X$, and we can therefore consider the \emph{set} $\pi_0\X(\pt)$ of $2$-isomorphism classes in $\X(\pt)$.

Each geometric point $\mathrm{pt}\to\X$ determines an element of $\pi_0\X(\pt)$ that we will denote by $[\pt\to\X]$.  To a morphism $U\to\X$ and a subset $A\subset \pi_0\X(\pt)$ we associate the subset
\[A_U=\{u\in U\mid [u\to\X]\in A\}\]
of $U$.  

\begin{definition}\label{UnderlyingSpaceDefinition}
Let $\X$ be a differentiable stack.  The \emph{underlying space} $\bar\X$ of $\X$ is the set $\pi_0\X(\pt)$ with the topology in which  $A\subset\bar\X$ is open if and only if $A_U\subset U$ is open for each morphism $U\to\X$.
\end{definition}

\begin{definition}\label{UnderlyingMapDefinition}
Let $f\colon\X\to\Y$ be a morphism of differentiable stacks.  Then the \emph{underlying map} $\bar f\colon\bar\X\to\bar\Y$ of $f$ is the function given by $[\mathrm{pt}\to\X]\mapsto[\mathrm{pt}\to\X\xrightarrow{f}\Y]$.
\end{definition}

The assignment $\X\mapsto\bar\X$, $f\mapsto\bar f$ is functorial in the sense that $\overline{\mathrm{Id}}=\mathrm{Id}$, that $\overline{gf}=\bar g\bar f$, and that $\bar f_1=\bar f_2$ if there is a $2$-morphism $f_1\Rightarrow f_2$.  Further properties of underlying maps are given in the following proposition.

\begin{proposition}\label{UnderlyingMapProposition}
Let $f\colon\X\to\Y$ be a morphism of differentiable stacks.  Then the underlying map $\bar f\colon\bar\X\to\bar\Y$ is continuous.  If $f$ is representable and has any of the following properties, then $\bar f$ also has that property:
\begin{quote}
injective, surjective, open, closed, topological embedding
\end{quote}
\end{proposition}

It is clear that if $\X=\mathrm{Mor}(-,X)$, then $\bar\X=X$.  The next proposition shows that in general $\bar\X$ is a topological quotient of an atlas $X$ for $\X$, and in particular that for global quotients ${[X/G]}$ the underlying space is just the topological quotient $X/G$.

\begin{proposition}\label{XBarTopologyProposition}
Let $X\to\X$ be an atlas for $\X$.  Then the underlying map $X\to\bar\X$ identifies $\bar\X$ as the quotient of $X$ by the relation $(s\times t)(X\times_\X X)\subset X\times X$.
\end{proposition}

The following corollary of Proposition~\ref{UnderlyingMapProposition} is useful in relating the topology of $\bar\X$ to that of manifolds equipped with a morphism into $\X$.

\begin{corollary}\label{InteriorClosureCorollary}
Let $\X$ be a differentiable stack and let $A\subset\bar\X$ be a subset of the underlying space.  Then for any open morphism $U\to\X$ we have
\[(\interior A)_U=\interior(A_U),\qquad (\cl A)_U=\cl(A_U).\]
\end{corollary}
\begin{proof}
The two claims are equivalent and we will only prove the first.  Under the underlying map $U\to\bar\X$, $A_U$ is the preimage of $A$ and $(\interior A)_U$ is the preimage of $\interior A$.  The result follows since, by Proposition~\ref{UnderlyingMapProposition}, $U\to\bar\X$ is open.
\end{proof}

The following basic lemma is central to the proofs of Propositions \ref{UnderlyingMapProposition} and \ref{XBarTopologyProposition}.

\begin{lemma}\label{AtlasTopologyLemma}
Let $X\to\X$ be an atlas.  Then $A\mapsto A_X$ defines a one-to-one correspondence between open subsets of $\bar\X$ and {saturated} open subsets of $X$.  Here, $S\subset X$ is \emph{saturated} if it contains all points of $X$ that are isomorphic in $\X(\pt)$ to a point in $S$; equivalently, $S$ is saturated if $ts^{-1}S=S$ in the diagram
\[\xymatrix{
X\times_\X X\ar[r]^-t\ar[d]_s & X\ar[d]_{}="1"\\
X\ar[r]^{}="2" & \X.\ar@{=>}"1";"2"
}\]
\end{lemma}
\begin{proof}
Let $A\subset\bar\X$ be open.  Then $A_X$ is certainly a saturated open subset of $X$.  Conversely, suppose $S\subset X$ is a saturated open subset and define $A\subset\bar\X$ to be $\{[s\to\X]\mid s\in S\}$.  Then since $S$ is saturated, $A_X=S$.  It remains to prove that $A$ is open.  Let $U\to\X$ be any morphism and consider the diagram
\[\xymatrix{
X\times_\X U\ar[r]^-{\pi_2}\ar[d]_{\pi_1}& U\ar[d]_{}="1"\\
X\ar[r]^{}="2"&\X.\ar@{=>}"1";"2"
}\]
It is clear that $A_U=\pi_2\pi_1^{-1}S$, which is open since $\pi_1$ is continuous and $\pi_2$ is open.  This concludes the proof.
\end{proof}

\begin{proof}[Proof of Proposition~\ref{UnderlyingMapProposition}.]
Let $U\to\X$ and $V\to\Y$ be morphisms from manifolds, with $V\to\Y$ a submersion.  Then it is simple to verify that in the commutative diagram
\[\xymatrix{
U\times_\Y V\ar[d]\ar[rr]^{\tilde f} && V\ar[d]_{}="1"\\
U\ar[r]&\X\ar[r]_f^{}="2" & \Y\ar@{=>}"1";"2"
}\]
we have $(\bar f^{-1}B)_{U\times_\Y V}=\tilde f^{-1}(B_V)$ for $B\subset\bar\Y$.  Taking $U$ and $V$ to be atlases, it follows that $U\times_\Y V\to\X$ is an atlas, and for $B\subset\bar\Y$ open, $(\bar f^{-1}B)_{U\times_\Y V}=\tilde f^{-1}(B_V)$ is also open, so that by Lemma~\ref{AtlasTopologyLemma} $(\bar f^{-1}B)$ is open.  This shows that $\bar f$ is continuous.

Now suppose that $f$ is representable and let $V\to\Y$ be an atlas, so that we have a 2-commutative diagram
\begin{equation}\label{PullbackDiagram}\xymatrix{
\X\times_\Y V\ar[d]\ar[r]^-{\tilde f} & V\ar[d]_{}="1"\\
\X\ar[r]_f^{}="2"&\Y.\ar@{=>}"1";"2"
}\end{equation}

Suppose that $f$ is injective.  If $\bar f$ is not injective, then there are points $x,y\in U$ such that  $x\to\X$, $y\to\X$ are not 2-isomorphic while $x\to\X\to\Y$, $y\to\X\to\Y$ are $2$-isomorphic; let $v\in V$ be such that there are $2$-morphisms $x\Rightarrow v$, $y\Rightarrow v$.  These morphisms give us two distinct points in $\X\times_\Y V$ whose image in $V$ is $v$, and this contradicts the injectivity of $f$.

If $f$ is surjective, then so is $\tilde f$, and since every point of $\bar\Y$ has the form $[v\to\Y]$ for some $v\in V$, it follows that $\bar f$ is surjective.

It is simple to verify that in diagram \eqref{PullbackDiagram} we have $\tilde f(A_{\X\times_\Y V})=(\bar f A)_V$ for any $A\subset\bar\X$.  It follows immediately that $\bar f$ is open or closed if $f$ is open or closed respectively.

Finally, suppose that $f$ is a topological embedding.  Then $\bar f$ is certainly injective.  Let $W\subset\bar\X$ be open.  We must show that there is an open $B\subset\bar\Y$ for which $B\cap\bar f\bar\X=\bar f W$.  The fact that $\tilde f$ is an embedding means that there is $W'\subset V$ open with $W'\cap\tilde f(\X\times_\Y V)=\tilde f(W_{\X\times_\Y V})$.  We may replace $W'$ with the saturated open set $W''=t(s^{-1}W')$, which --- since $\tilde f(\X\times_\Y V)$ and $\tilde f(W_{\X\times_\Y V})$ are themselves saturated --- satisfies $W''\cap\tilde f(\X\times_\Y V)=\tilde f(W_{\X\times_\Y V})$.  It follows that $B\cap\bar f\bar\X=\bar f W$, where $B\subset\bar\Y$ is the open subset corresponding to $W''$.  Consequently $\bar f$ is a topological embedding.
\end{proof}

\begin{proof}[Proof of Proposition~\ref{XBarTopologyProposition}.]
Taking preimages under the underlying map $X\to\bar\X$ precisely realises the correspondence of Lemma~\ref{AtlasTopologyLemma}.
\end{proof}

\subsection{Substacks and the underlying space.}\label{SubstackSubsection}
We now recall the notion of substack of a stack and show how subsets of $\bar\X$ correspond to a certain kind of substack of $\X$ that we call \emph{subsets} of $\X$.  This correspondence will be used in an essential way in \cite{\OrbifoldMSW} to define the stable and unstable manifolds of a critical point of a Morse function on a Riemannian Deligne-Mumford stack.

\begin{definition}
Let $\mathcal{S}$ be a site and let $\X$ be a stack over $\mathcal{S}$, regarded as a weak functor $\mathcal{S}^\mathrm{op}\to\mathrm{Groupoids}$.  Then recall from \cite[3.9]{\Noohi} that a \emph{substack} $\Y$ of $\X$ is a full saturated subfunctor of $\X$ that is also a stack.
\end{definition}

The definition means that $\Y(W)$ is a full saturated subgroupoid of $\X(W)$ for each $W\in\mathrm{Ob}(\mathcal{C})$, that $\Y$ inherits its structure as a weak functor from that of $\X$, and that $\Y$ is itself a stack.  (A subcategory $\mathcal{C}\subset\mathcal{D}$ is \emph{full} if every morphism in $\mathcal{D}$ between objects in $\mathcal{C}$ is itself in $\mathcal{C}$, and it is \emph{saturated} if every object in $\mathcal{D}$ that is isomorphic to one in $\mathcal{C}$ is itself in $\mathcal{C}$.)

For the purposes of the next definition and the discussion that follows we will distinguish between a manifold $U$ and the stack $\underline U$ that it represents.

\begin{definition}\label{SubsetDefinition}
Let $U$ be a manifold and let $V\subset U$ be a subset.  We can then form the substack ${\underline V}_U\subset \underline U$ for which ${\underline V}_U(W)\subset \underline U(W)$ is the set of maps $W\to U$ with image in $V$.   We will call the substack ${\underline V}_U\subset\underline U$ a \emph{subset of $U$}, and ${\underline V}_U\to\underline U$ \emph{the inclusion of a subset of $U$}.  Just as we usually write $\underline U$ as $U$, so we will usually write $\underline V_U$ as $V$, even though $\underline V_U$ depends on $U$.
\end{definition}

\begin{definition}
A substack $\Y$ of $\X$ is called a \emph{subset} if, for all $U\to\X$, the morphism $U\times_\X\Y\to U$ is equivalent to the inclusion of a subset.
\end{definition}

Each substack $\Y$ of $\X$ determines a subset $Y=\{[\pt\to\X]\mid \pt\to\X\in\Y(\pt)\}$ of $\bar\X$.  In the case that $\Y$ is a differentiable stack, $Y$ is the image of the underlying map $\bar\Y\to\bar\X$.

\begin{proposition}\label{SubsetProposition}
The assignment $\Y\mapsto Y$ determines a correspondence between subsets of $\X$ and subsets of $\bar\X$.  This restricts to a correspondence between open subsets of $\X$ and open subsets of $\bar\X$.
\end{proposition}

\begin{proof}
Let $\Y$ be a subset of $\X$ and let $Y$ be the corresponding subset of $\bar\X$.  We claim that for each $W$, $\Y(W)\subset\X(W)$ is the collection of those $W\to\X$ with the property that $[w\to\X]\in Y$ for each $w\in W$.  Thus, $\Y$ is determined by $Y$.

To prove the claim, first note that each $\Y(W)$ is contained in the stated subcollection of $\X(W)$.  Now let $X\to\X$ be an atlas for $\X$, so that $\Y\times_\X X\to X$ is equivalent to the inclusion of a subset; it is clear by considering points that the subset in question is $Y_X$.  We have a cartesian diagram
\[\xymatrix{
Y_X\ar[r]\ar[d] & X\ar[d]^{}="1"\\
\Y\ar[r]_{}="2"& \X.\ar@{=>}"1";"2"
}\]
Now let $U\to\X$ have the property that each $[u\to\X]$ is in $Y$.  Then $U\times_\X X\to X$ factors through $Y_X$, so that (since $\Y(U\times_\X X)\subset \X(U\times_\X X)$ is saturated) $U\times_\X X\to U\to\X$ factors through $\Y$.  Since $\Y$ is itself a stack, and $\Y(U)\subset\X(U)$ is full and saturated, and $U\times_\X X\to U$ is a surjective submersion, it follows that $U\to\X$ factors through $\Y$ as required.

We must now show that given $Y\subset\bar\X$, the specification
\[\Y(W)=\{W\to\X\mid[w\to\X]\in Y\mathrm{\ for\ all\ }w\in W\}\]
determines a substack $\Y$ of $\X$ that is also a subset, and that $\Y$ determines the original subset $Y\subset\bar\X$.

First, it is clear that the specification does determine a full, saturated subfunctor $\Y$ of $\X$.  Second, to check that $\Y$ is a stack we must verify that, given $W=\bigcup W_i$ and $W\to\X$ with each $W_i\to\X$ factoring through $\Y$, then $W\to\X$ itself factors through $\Y$; but this is again clear.  Finally, it is simple to verify that for any $U\to\X$, the induced $U\times_\X\Y\to U$ is equivalent to the inclusion of $\{u\in U\mid[u\to\X]\in Y\}$.

By the previous paragraph, if $Y\subset\bar\X$ is open and $\Y\to\X$ is the corresponding substack, then for $U\to\X$, $\Y\times_\X U\to U$ is equivalent to the inclusion of the open subset $Y_U$, so that $\Y$ is indeed an open substack of $\X$.  Conversely, if $\Y\to\X$ is an open inclusion, then so is $\bar\Y\to\bar\X$ by Proposition~\ref{UnderlyingMapProposition}, so that the resulting $Y=\mathrm{Im}(\bar\Y\to\bar\X)$ is indeed open.
\end{proof}

\subsection{Coarse moduli spaces.}\label{CoarseModuliSubsection}

We now compare the material presented in the last two subsections, which were concerned with differentiable stacks, with some existing results of Noohi on topological stacks \cite[4.3]{\Noohi}.  We will recall, and concentrate on, the notion of coarse moduli space and its relation to the underlying space.

We will see that our definition of underlying space can be applied to topological stacks $\X$ just as well as differentiable stacks.  The resulting space $\bar\X$ coincides with a space $\X_\mathrm{mod}$ defined by Noohi.  However, the significance of the underlying space is quite different in the two contexts: Noohi showed that $\X_\mathrm{mod}$ is always a coarse moduli space for $\X$, but we will see that not all differentiable stacks admit a coarse moduli space at all, and that even when a coarse moduli spaces does exist it need not coincide with the underlying space.

Let $\mathrm{Top}$ denote the site of compactly generated topological spaces and continuous maps, and recall from \cite[\S 13]{\Noohi} that a topological stack is a stack $\X$ over $\mathrm{Top}$ which admits a representable {\bf LF} surjection $X\to\X$ from a space $X$.  Here {\bf LF} is a class of local fibrations, and we will assume that all {\bf LF} maps are open.  Given a topological stack $\X$, Noohi defined the space $\X_\mathrm{mod}$ to be the set $\pi_0\X(\mathrm{pt})$ equipped with the topology whose open sets are $\mathfrak{U}_\mathrm{mod}\subset\X_\mathrm{mod}$ for open substacks $\mathfrak{U}\subset\X$.  Now note that, by replacing manifolds with spaces and submersions with {\bf LF} maps, the material of subsections \S\ref{UnderlyingSpaceSubsection} and \S\ref{SubstackSubsection} holds for topological stacks in place of differentiable stacks, without any further change.  We can therefore define the \emph{underlying space} $\bar\X$ of a topological stack $\X$.

\begin{proposition}
The underlying space $\bar\X$ of a topological stack $\X$ is canonically isomorphic to $\X_\mathrm{mod}$.
\end{proposition}
\begin{proof}
Both $\bar\X$ and $\X_\mathrm{mod}$ are given as sets by $\pi_0\X(\mathrm{pt})$.  We need only check that the topologies coincide, but this is immediate from Proposition~\ref{SubsetProposition}.
\end{proof}

\begin{definition}
If $\X$ is a stack over a site $\mathcal{S}$, then a \emph{coarse moduli space} for $\X$ is an object $\X_\mathrm{mod}$ of $\mathcal{S}$, together with a \emph{coarse moduli morphism} $\mathrm{mod}\colon\X\to\X_\mathrm{mod}$ that has the following property:
\begin{quote}
Every morphism $\X\to U$, for $U$ a representable stack, factors uniquely through $\mathrm{mod}$:
\[\xymatrix{
\X\ar[rr]\ar[rd]_{\mathrm{mod}} && U\\
&\X_{\mathrm{mod}}\ar[ur]&
}\]
\end{quote}
A coarse moduli space need not exist, though if it does then it is unique up to isomorphism.
\end{definition}

\begin{proposition}[{\cite[4.15]{\Noohi}}]\label{CoarseModuliProposition}
Let $\X$ be a topological stack.  There is a \emph{canonical morphism}
\[\mathrm{mod}\colon\X\to\bar\X\]
which is a coarse moduli morphism, so that $\X_\mathrm{mod}$ is a coarse moduli space for $\X$.
\end{proposition}

We have established that the notion of underlying space makes sense in the topological setting and coincides with the notion of coarse moduli space.  Now we will show using two examples that in the differentiable setting this is far from true.

\begin{example}
Consider the differentiable stack $\mathfrak{G}=[\R^n/\mathrm{GL}(n,\R)]$, where $\mathrm{GL}(n,\R)$ acts on $\R^n$ in the tautological manner.
\begin{itemize}
\item The underlying space $\bar{\mathfrak{G}}$ is the topological quotient $\R^n/\mathrm{GL}(n,\R)$.  This is the two-point space $\{\mathcal{O}_0,\mathcal{O}_{\neq 0}\}$, where $\mathcal{O}_0$, $\mathcal{O}_{\neq 0}$ are the orbits in $\R^n$ of the zero vector and of the nonzero vectors respectively.  This is a non-Hausdorff space whose only proper open subset is $\{\mathcal{O}_{\neq 0}\}$.
\item $\mathfrak{G}$ has coarse moduli space $\mathfrak{G}_\mathrm{mod}=\pt$, and the constant morphism $\mathfrak{G}\to\pt$ is a coarse moduli morphism.  This is because manifolds are Hausdorff, and so any $\mathrm{GL}(n,\R)$-invariant morphism from $\R^n$ to a manifold must be constant.
\end{itemize}
In this instance the coarse moduli space exists but is different from the underlying space.
\end{example}

\begin{example}
Consider the differentiable stack $\mathfrak{C}=[\C/ \Z_2]$, where $\Z_2$ acts on $\C$ by $z\mapsto -z$.  There is a unique morphism
\[m\colon\mathfrak{C}\to\C\]
for which the composite $\C\to\mathfrak{C}\xrightarrow{m}\C$ is $z\mapsto z^2$.  Taking underlying spaces, $m$ becomes the homeomorphism $\bar m\colon \C/\Z_2\to\C$, $\pm z\mapsto z^2$.  

We could just as well regard $m\colon\mathfrak{C}\to\C$ as a morphism of topological, complex differentiable, or algebraic stacks.  In each of these cases $m$ makes $\C$ a coarse moduli space of $\mathfrak{C}$.  However, in the differentiable case $m$ is \emph{not} a coarse moduli morphism, and $\mathfrak{C}$ does not in fact have a coarse moduli space, as we will now show.

To see that $m\colon\mathfrak{C}\to\C$ is not coarse moduli consider the morphism $\mathfrak{C}\to\R$ that becomes  $z\mapsto|z|^2$ when composed with $\C\to\mathfrak{C}$.  If $m$ were a coarse moduli morphism, then $\mathfrak{C}\to\R$ would have to factor through $m$ using the map $\C\to\R$, $z\mapsto|z|$, which is not smooth.

We will now show that $\mathfrak{C}$ does not have a coarse moduli space.  Suppose that there were a coarse moduli morphism $\mathrm{mod}\colon\mathfrak{C}\to C$.  Then there would be a factorization:
\[\xymatrix{\mathfrak{C}\ar[rd]_{\mathrm{mod}}\ar[rr]^m && \C\\ & C\ar[ur]_n &}\]
The composition $\C\to\mathfrak{C}\to C$ is $\Z_2$-invariant and surjective, and when composed with $n$ becomes $z\mapsto z^2$.  Thus $\overline{\mathrm{mod}}$ is surjective, and so $n$ is a homeomorphism.  Since $\C\to\mathfrak{C}\to C$ is $\Z_2$-invariant, its derivative at $0$ vanishes; we claim that the derivative of $n$ at $0$ also vanishes.  By considering Taylor approximations at $0$, we will then obtain a contradiction since the composite $\C\to\mathfrak{C}\xrightarrow{\mathrm{mod}} C\xrightarrow{n}\C$ is $z\mapsto z^2$, which cannot be the composite of two maps whose derivatives at $0$ vanish.

We now show that the derivative of $n$ at $0$ must vanish.  Since $m$ is not coarse moduli we know that $n$ cannot be a diffeomorphism.  However, it is clear that $n|\colon C-\{0\}\to\C-\{0\}$ is a diffeomorphism, and so the derivative of $n$ at $0$ must be singular.  Consider the action of $T^1$ on $\C$ given by $z\mapsto e^{i\theta}z$.  This commutes with the $\Z_2$-action, and so determines a weak action on $\mathfrak{C}=[\C/\Z_2]$, and consequently an action on $C$.  Moreover, $m$ is equivariant if we allow $T^1$ to act on the range by $z\mapsto e^{2i\theta}z$.  It follows that $n$ is also equivariant.  Now, since the derivative of $n$ at $0$ is singular with $T^1$-invariant image it must vanish completely as claimed.
\end{example}

\section{Differentiable Deligne-Mumford Stacks.}\label{DDMStacksSection}

In this section we will recall the definition of differentiable Deligne-Mumford stacks and establish some basic facts.  In \S\ref{DDMDefinitionSection} we define such stacks and discuss the connection with proper \'etale groupoids and orbifolds.  In \S\ref{OrbifoldChartsSection} we show that differentiable Deligne-Mumford stacks admit orbifold-charts and we establish a direct link with orbifolds via the underlying space.  Then in \S\ref{ParacompactnessSection} we will give results on paracompactness and partitions of unity.  Finally in \S\ref{TechnicalSection} we gather some miscellaneous results which will be used in later sections.

\subsection{The definition of differentiable Deligne-Mumford stacks.}\label{DDMDefinitionSection}

\begin{definition}
Recall that a \emph{differentiable Deligne-Mumford stack} is a differentiable stack $\X$ which admits an \'etale surjection $X\to\X$ and whose diagonal $\Delta\colon\X\to\X\times\X$ is proper.  
\end{definition}

The existence of an \'etale atlas on $\X$ ensures that the automorphism groups of the points of $\X$ are discrete, while the properness condition ensures that they are finite.  Properness also serves as an important Hausdorff-type condition that is much stronger than simply requiring that automorphism groups be finite.  Indeed, there are stacks which admit an \'etale atlas, whose points all have trivial automorphism groups, but which are not Deligne-Mumford.  Moreover, we will see in \S\ref{CounterexamplesSection} that several important results that hold for Deligne-Mumford stacks will fail for stacks that merely admit an \'etale atlas and have finite automorphism groups.

\begin{example}
Let $G$ be a Lie group acting properly and almost-freely on a manifold $M$.  Then the stack $\X=[M/G]$ is differentiable Deligne-Mumford.

Properness of the action means that $G\times M\to M\times M$ is proper, which exactly means that $\Delta\colon\X\to\X\times\X$ is proper.  Since the action is proper we may take a slice $U_m\hookrightarrow M$ to the $G$-action at any point $m\in M$.  Almost-freeness of the $G$-action means that the isotopy group $G_m$ is finite and so $G\times U_m\to (G\times U_m)/G_m\hookrightarrow M$ is \'etale.  This shows that $U_m\to\X$ is \'etale, and so $\X$ admits an  \'etale atlas.
\end{example}

If $\X$ is differentiable Deligne-Mumford and $X\to\X$ is an \'etale atlas then $X\times_\X X\rightrightarrows X$ is an proper \'etale Lie groupoid.  Conversely, every proper \'etale Lie groupoid represents a differentiable Deligne-Mumford stack.  Thus the sub-2-category of stacks over $\mathrm{Diff}$ whose objects are the differentiable Deligne-Mumford stacks is equivalent to the $2$-category of proper \'etale Lie groupoids with Morita equivalences among them weakly inverted.  The relationship between orbifolds and this second $2$-category is well-known \cite[\S 3]{\Moerdijk}, \cite[\S1.4]{\ALR}, at least if we restrict to effective orbifolds and effective groupoids.  It is therefore reasonable to work with differentiable Deligne-Mumford stacks in place of orbifolds or proper \'etale Lie groupoids.

\subsection{Orbifold-charts.}\label{OrbifoldChartsSection}

\begin{definition}
Let $\X$ be a differentiable Deligne-Mumford stack.
\begin{enumerate}
\item An \emph{orbifold-chart} on $\X$ is an open embedding $[U/G]\to\X$, where $G$ is a finite group acting on a manifold $M$.
\item Let $x$ be a $2$-isomorphism class of points in $\X$.  An \emph{orbifold-chart at $x$} is an orbifold-chart $[U/G]\to\X$ together with a distinguished point $x_U\in U$ such that $[x_U\to\X]=x$ and $Gx_U=x_U$.
\item An orbifold chart $[U/G]\to\X$ at $x$ is \emph{linear} if $U$ is an open subset of some linear $G$-representation and $x_U={0}$.
\end{enumerate}
\end{definition}

The following is an immediate consequence of \cite[3.4]{\Moerdijk}.
\begin{proposition}\label{OrbifoldChartsProposition}
Let $x$ be a $2$-isomorphism class of points in a differentiable Deligne-Mumford stack $\X$.  Then there is an orbifold-chart $[U/G]\to\X$ at $x$, which we may assume to be linear.
\end{proposition}

\begin{corollary}
The underlying space $\bar\X$ of a differentiable Deligne-Mumford stack has a natural orbifold-structure in the sense of \cite[Appendix]{\ChenRuanGromovWitten}.
\end{corollary}
\begin{proof}
First, $\bar\X$ is a second-countable Hausdorff space as we will see in the next section.  Second, an orbifold-chart $\iota_x\colon[U_x/G_x]\to\X$ at $x\in\bar\X$ induces a uniformizing system $(U_x,G_x,\bar{\iota_x})$ on the image of $\bar{\iota_x}$ in $\bar\X$.  Any two uniformizing systems obtained in this way are compatible.  Thus $\bar\X$ has the structure of an orbifold.  
\end{proof}

\subsection{Paracompactness and partitions of unity.}\label{ParacompactnessSection}

Any locally compact, second countable Hausdorff space (for example, any manifold) is paracompact, and for any open cover of such a space there is a countable partition of unity subordinate to that cover.  Manifolds are locally compact, second countable Hausdorff spaces, and in this case the partition of unity can be taken to consist of smooth functions (see \cite[1.9, 1.11]{\Warner}).  In this section we will prove analogous results for differentiable Deligne-Mumford stacks.

\begin{proposition}\label{DMUnderlyingSpaceProposition}
Let $\X$ be a differentiable Deligne-Mumford stack.  Then $\bar\X$ is a locally-compact second countable Hausdorff space.  In particular, $\bar\X$ is paracompact \cite[1.9]{\Warner} and, by Urysohn's Metrization Theorem, metrizable.
\end{proposition}
\begin{proof}
Let $\pi\colon X\to\X$ be an \'etale atlas for $\X$.  Proposition~\ref{XBarTopologyProposition} tells us that $\bar\X$ is homeomorphic to the quotient of $X$ by the equivalence relation $(s\times t)(X\times_\X X)\subset X\times X$.  

Since $s\times t\colon X\times_\X X\to X\times X$ is proper and $X\times X$ is locally compact, $s\times t$ has closed image.  It follows that $\bar\X$ is Hausdorff.  Let $\{U_i\}$ be a countable basis for the topology of $X$.  Then $\{ts^{-1}U_i\}$ is a countable family of \emph{invariant} open subsets of $X$, and any other open invariant subset of $X$ is a union of elements of $\{ts^{-1}U_i\}$.  Thus, by Proposition~\ref{XBarTopologyProposition}, $\{\bar\pi(ts^{-1}U_i)\}$ is a countable basis for the topology on $\bar\X$.  Finally, any point $x\in\bar\X$ lies in the image of the open embedding $\overline{[X/G]}\to\bar\X$ underlying some orbifold-chart $[X/G]\to\X$.  Since $\overline{[X/G]}$ is locally compact, the same holds for $\bar\X$.
\end{proof}

\begin{definition}
Given a morphism $f\colon\X\to\R$, the \emph{support of $f$}, denoted $\supp f$, is the subset $\cl\{x\in\bar\X\mid\bar f(x)\neq 0\}$ of $\bar\X$.  In other words $\supp f$ is just the support of $\bar f$.  By Corollary~\ref{InteriorClosureCorollary} we have $(\supp f)_U=\supp(f\circ\pi_U)$ for any submersion $\pi\colon U\to\X$.
\end{definition}

\begin{theorem}[Existence of partitions of unity.]\label{PartitionsOfUnityTheorem}
Let $\X$ be differentiable Deligne-Mumford and let $\{U_\alpha\}$ be an open cover of $\bar\X$.  Then there are morphisms $\phi_i\colon\X\to\R$ for $i=1,2,\ldots$ such that the $\bar\phi_i$ are a partition of unity on $\bar\X$, subordinate to $\{U_\alpha\}$, with each $\supp\phi_i$ compact.
\end{theorem}
\begin{proof}
Using the fact that $\bar\X$ is a locally compact topological space, one can prove this result exactly as one proves \cite[Theorem 1.11]{\Warner}, after establishing the two facts below.

First, we must show that given $x\in\bar\X$ and a neighbourhood $U$ of $x$, there is $\psi\colon\X\to\R$ with $\image\bar\psi\subset[0,1]$, with $\supp\psi\subset U$ compact and with $\bar\psi=1$ in a neighbourhood of $x$.  To see this, take an orbifold-chart $[M/G]\to\X$ at $x$ and a function $p\colon M\to\R$ with $p=1$ in a neighbourhood of $x_M$, $p\geqslant 0$, and $\supp p\subset V_M$ compact.  By averaging we may assume that $p$ is $G$-invariant and write $p\colon[M/G]\to\R$ for the corresponding map.  Then using Lemma~\ref{ExtensionByZeroExistsLemma} in \S\ref{TechnicalSection} we may extend $p$ to the required morphism $\X\to\R$.

Second, we must check that, given $\psi_i\colon\X\to\R$ for $i=1,2,\ldots$ with $\supp\psi_i$ locally finite, with $\bar\psi_i\geqslant 0$, and with some $\bar\psi_i(x)$ nonzero for each $x\in\bar\X$, then there is a function $\psi\colon\X\to\R$ with $\bar\psi=\sum\bar\psi_i$.  To see this, let $\pi_U\colon U\to\X$ be any surjection, and note that the $\supp(\psi_i\circ\pi_U)$ form a locally finite family so that we may define $\psi_U=\sum\psi_i\circ\pi_U$.  It is clear that the $\psi_U$ satisfy the required conditions for the existence of a morphism $\psi\colon\X\to\R$ with $\psi\circ\pi_U=\psi_U$.  Since for each $u\in U$ we have $\psi_U(u)=\sum\psi_i(u)$, it follows that $\bar\psi=\sum\bar\psi_i$.
\end{proof}

\subsection{Some technical results.}\label{TechnicalSection}

This section gives several technical results for differentiable Deligne-Mumford stacks that will be important for applications later in the paper.

\begin{proposition}
Suppose we have a diagram of differentiable stacks
\[\xymatrix{\Y
\ar@/^1pc/[rr]^{f}_{}="1"
\ar@/_1pc/[rr]_{g}^{}="2"\ar@{=>}"1";"2"^{\lambda,\mu}
&& \X}\]
where $\X$ admits an \'etale atlas and $\bar\Y$ is connected.  Fix a point $y\to\Y$.  Then $\lambda=\mu$ if and only if $\lambda|_y=\mu|_y$.
\end{proposition}
\begin{proof}
We may suppose that $f=g$ and that $\mu=\Id$.  Note that for any point $\pt\to\Y$ the question of whether $\lambda|_\pt=\Id$ depends only on $[\pt\to\Y]$.  We will prove that the subset
\[\mathfrak{I}=\{[\pt\to\Y]\mid \lambda|_\pt=\Id\}\]
of $\bar\Y$ is both open and closed.

First, let $X\to\X$ be an \'etale atlas and let $e\colon X\to X\times_\X X$ denote the identity of the groupoid $X\times_\X X\rightrightarrows X$.  Then the image of $e$ is both open and closed.  To see this note that the composition of $e$ with either projection map $X\times_\X X\to X$ is the identity.  It immediately follows that $e$ is closed, and since the projection maps are \'etale it follows that $e$ is \'etale and in particular is open.

Now consider the diagram
\[\xymatrix{
X\times_\X\Y\ar[r]\ar[d] & X\ar[d]_{}="1"\\
\Y\ar[r]_f^{}="2" & \X\ar@{=>}"1";"2"
}\]
The $2$-morphism $\lambda$ induces a map $l\colon X\times_\X\Y\to X\times_\X X$, and $\lambda|_{u}=\Id$ for $u\in X\times_\X\Y$ if and only if $l(u)$ lies in the image of $e$.  The last paragraph shows that
\begin{eqnarray*}
\mathfrak{I}_X
&=&\{u\in X\times_\X\Y\mid \lambda|_u=\Id\}\\
&=&\{u\in X\times_\X\Y\mid l(u)\in\image(e)\}
\end{eqnarray*}
is both open and closed.  It follows that $\mathfrak{I}$ is itself both open and closed, and this completes the proof.
\end{proof}

\begin{lemma}[Extension by zero.]\label{ExtensionByZeroExistsLemma}
Let $i\colon\Y\to\X$ be an open embedding of differentiable stacks.  Suppose given $\phi\colon\Y\to\R$ with $\supp\phi$ compact.  Then there is a unique $\tilde\phi\colon\X\to\R$ such that $\tilde\phi\circ i=\phi$ and such that the map underlying ${\tilde\phi}$ vanishes on $\bar\X-\bar i\bar\Y$.
\end{lemma}
\begin{proof}
Let $\pi_A\colon A\to\X$ be an atlas for $\X$ and $\pi_B\colon B\to\Y$ the induced atlas for $\Y$, so that we have a $2$-commutative diagram 
\[\xymatrix{
B\ar[r]^{\tilde i}\ar[d]_{\pi_B} & A\ar[d]^{\pi_A}\\
\Y\ar[r]_i & \X
}\]
with $\tilde i$ an open embedding.  Since $\supp\phi\subset\bar\Y$ is compact, $\bar i(\supp\phi)\subset\bar\X$ is compact and therefore closed, so that $\tilde i(\supp(\phi\circ\pi_B))$ is closed.  We may therefore extend $\phi\circ\pi_B$ by zero to obtain $\tilde\phi_A\colon A\to\R$ with $\tilde\phi_A\circ\tilde i=\phi\circ\pi_B$ and $\tilde\phi_A=0$ on $A-\tilde i B$.  It is clear that the two compositions $A\times_\X A\rightrightarrows A\xrightarrow{\tilde\phi_A}\R$ coincide, so that there is $\tilde\phi\colon\X\to\R$ with $\tilde\phi\circ\pi_A=\tilde\phi_A$.  By construction, $\tilde\phi\circ i=\phi$ and $\bar{\tilde\phi}=0$ on $\bar\X-\bar i\bar\Y$.
\end{proof}

\begin{lemma}\label{BumpFunctionsLemma}
Let $\X$ be differentiable Deligne-Mumford.  Let $K\subset\bar\X$ be compact and let $U$ be an open neighbourhood of $K$.  Then we may find $f\colon\X\to\R$ such that $\bar f=1$ in a neighbourhood of $K$ and such that $\supp f$ is compact and contained in $U$.
\end{lemma}
\begin{proof}
Take a partition of unity $\phi_1,\phi_2,\ldots$ subordinate to the cover $\{U,\bar\X\setminus K\}$ and set 
\[f=\sum\phi_i\]
where the sum is taken over those $i$ for which $\supp\phi_i$ is contained in $U$.  This is the required function.
\end{proof}

\begin{definition}
Let $\X$ be a differentiable stack.  A family of morphisms $\alpha_i\colon A_i\to\X$ is called \emph{locally finite} if for each $U\to\X$ and each $u\in U$ there is a neighbourhood $V$ of $u$ such that $V\times_\X A_i$ is non-empty for only finitely many $i$.  This is if and only if the underlying maps $\bar\alpha_i\colon A_i\to\bar\X$ form a locally finite family.
\end{definition}

\begin{proposition}\label{SpecialCoverProposition}
Let $\X$ be differentiable Deligne-Mumford.  We can find a countable locally finite family of \'etale morphisms $s_l\colon S_l\to\X$ from open subsets of $\R^n$, together with open subsets $T_l\subset S_l$ for which each $\cl T_l$ is compact and $\bigsqcup s_l\colon\bigsqcup T_l\to\X$ is surjective.  If we wish we may assume that the $s_l\colon S_l\to\X$ are obtained from orbifold-charts $[S_l/G_l]\to\X$.
\end{proposition}
\begin{proof}
For each $x\in\bar\X$ choose a linear orbifold-chart $[U_x/G_x]\to\X$ at $x$, so that $U_x$ is an open neighbourhood of the origin in some $n$-dimensional representation of $G_x$.  Choose an invariant open neighbourhood $V_x\subset U_x$ of $0$ with $\cl_{U_x} V_x$ compact.  Set $A_x=\overline{[U_x/G_x]}$, $B_x=\overline{[V_x/G_x]}$, so that the two covers $\{A_x\}$, $\{B_x\}$ of $\bar\X$ satisfy the hypotheses of Lemma~\ref{CoverLemma} below.  (Note that the overlines refer to underlying spaces, not to closures.)  Applying the lemma, we obtain countable locally finite covers $\{A'_{x_l}\}$, $\{B'_{x_l}\}$ of $\bar\X$, with $A'_{x_l}\subset A_{x_l}$, $B'_{x_l}\subset B_{x_l}$, and $\cl_{\bar\X} B'_l\subset A'_l$ compact.  Now $\{A'_{x_l}\}$ and $\{B'_{x_l}\}$ determine open subsets
\begin{gather*}
U'_{x_l}=\{u\in U_{x_l}\mid [u\to\X]\in A'_{x_l}\}=(A'_{x_l})_{U_{x_l}},\\
V'_{x_l}=\{v\in V_{x_l}\mid [v\to\X]\in B'_{x_l}\}=(B'_{x_l})_{U_{x_l}},
\end{gather*}
and $\cl_{U'_{x_l}}V'_{x_l}=(\cl_{\bar\X}B'_{x_l})_{U'_{x_l}}$ is compact, since it is the preimage under the proper map $U'_{x_l}\to\overline{[U'_{x_l}/G_{x_l}]}=A'_{x_l}$ of the compact $\cl_{\bar\X}B'_{x_l}$.  Thus the morphisms $U'_{x_i}\to\X$, together with the subsets $V'_{x_i}\subset U'_{x_i}$, are the required data.
\end{proof}

\begin{lemma}\label{CoverLemma}
Let $X$ be a locally-compact Hausdorff space and let $\mathcal{U}=\{U_\alpha\}_{\alpha\in A}$, $\mathcal{V}=\{V_\alpha\}_{\alpha\in A}$ be open covers of $X$, such that $\cl V_\alpha\subset U_\alpha$ for each $\alpha$.  Then there are countable locally finite refinements $\mathcal{U}'=\{U'_\beta\}_{\beta\in B}$ and $\mathcal{V}'=\{V'_\beta\}_{\beta\in B}$ of $\mathcal{U}$ and $\mathcal{V}$ such that $\cl V'_\beta\subset U'_\beta$ for each $\beta\in B$.
\end{lemma}
\begin{proof}
This is a mild modification of the proof of \cite[1.9]{\Warner}.  As in that proof, let $G_i$ for $i=1,2,\ldots$ be a sequence of open subsets of $X$ such that
\begin{gather*}
\bigcup G_i=X,\\
\cl G_i\subset G_{i+1},\\
\cl G_i\mathrm{\ compact}.
\end{gather*}
In what follows, $G_i$ for $i\leqslant 0$ should be taken to be the empty set $\emptyset$.

For each $i$, the $V_\alpha\cap(G_{i+1}-\cl G_{i-2})$ form an open cover of the compact set $\cl G_i-G_{i-1}$, and we may therefore choose finitely many $\alpha^i_l$ such that the $V_{\alpha^i_l}\cap(G_{i+1}-\cl G_{i-2})$ cover $\cl G_i-G_{i-1}$.  Set
\begin{gather*}
B=\{\alpha^i_l\},\\
V'_{\alpha^i_l}=V_{\alpha^i_l}\cap(G_{i+1}-\cl G_{i-2}),\\
U'_{\alpha^i_l}=U_{\alpha^i_l}\cap(G_{i+2}-\cl G_{i-3}).
\end{gather*}
Then
\begin{eqnarray*}
\cl V'_{\alpha^i_l}
&\subset& \cl V_{\alpha^i_l}\cap(\cl G_{i+1}- G_{i-2})\\
&\subset& U_{\alpha^i_l}\cap(G_{i+2}-\cl G_{i-3})\\
&=&U'_{\alpha^i_l},
\end{eqnarray*}
and the covers are locally finite since $(G_{i+2}-\cl G_{i-3})\cap (G_{j+2}-\cl G_{j-3})\neq\emptyset$ only when $|i-j|\leqslant 4$.
\end{proof}

\section{Morse Functions.}\label{MorseFunctionSection}

In this section we discuss Morse functions on differentiable Deligne-Mumford stacks.  In \S\ref{MorseFunctionDefinitionSection} we define Morse functions, their critical points, and the index and co-index of critical points.  This generalises from manifolds the notion of Morse function, critical points, and index of a critical point.  However, the index and co-index of a Morse function on a differentiable Deligne-Mumford stack are much richer quantities than the index of a critical point of a Morse function on a manifold, and this is crucial in a correct formulation of the Morse inequalities.   In \S\ref{MorseLemmaSection} we prove a Morse Lemma for differentiable Deligne-Mumford stacks.  Finally, in \S\ref{InertiaStackSection} we show how a Morse function on $\X$ gives rise to a Morse function on the inertia stack $\Lambda\X$.

\subsection{Morse functions.}\label{MorseFunctionDefinitionSection}

\begin{definition}\label{TangentSpaceDefinition} Let $\X$ be a differentiable Deligne-Mumford stack.
\begin{enumerate}
\item Let $x$ be a $2$-isomorphism class of points in $\X$.  Define the \emph{automorphism group} $\aut_x$ of $x$ and the \emph{tangent space} $T_x\X$ of $x$ as follows.  Choose an \'etale morphism $U\to\X$ with a point $u\in U$ that represents $x$.  Then the automorphism group $\aut_u$ of $u$ in the groupoid $U\times_\X U\rightrightarrows U$ is finite and acts linearly on $T_uU$.  Set $\aut_x=\aut_u$ and $T_x\X=T_uU$.  We regard $T_x\X$ as a representation of $\aut_x$.
\item Let $f\colon\X\to\R$ be a morphism.  We define $d_xf\colon T_x\X\to\R$ to be the derivative at $u$ of the composite $U\to\X\xrightarrow{f}\R$.  It is an $\aut_x$-invariant linear map.
\item If $d_xf=0$ then we say that $x$ is a \emph{critical point of $f$} and that $f(x)$ is a \emph{critical value of $f$}.  When $x$ is a critical point of $f$ we define the \emph{Hessian of $f$ at $x$}, which is a symmetric bilinear form
\[H_{f,x}\colon T_x\X\times T_x\X\to\R,\]
to be the Hessian at $u$ of the composite $U\to\X\xrightarrow{f}\R$.  We say that $x$ is a \emph{degenerate critical point of $f$} if $H_{f,x}$ is singular; otherwise we say that $x$ is a \emph{nondegenerate critical point}.
\end{enumerate}
\end{definition}

When $\X$ is a manifold the automorphism group of any point is trivial and Definition~\ref{TangentSpaceDefinition} simply gives us the usual definition of tangent space, derivative of $f$ and Hessian of $f$.

\begin{note}\label{TangentSpaceNote}
When $\X$ is not a manifold then $\aut_x$, $T_x\X$, $d_xf$ and $H_{f,x}$ are dependent on the choice of $U\to\X$ and $u\in U$.  If we make a different choice of $U$ and $u$ then $\aut_x$ will only change up to isomorphism, while $T_x\X$ will only change up to an equivariant isomorphism that intertwines the two choices of $d_xf$ and the two choices of $H_{f,x}$.  Most importantly, the isomorphisms themselves are not unique; they can vary by an inner automorphism of $\aut_x$ and the corresponding isomorphism of $T_x\X$.  For notational convenience we will largely ignore these ambiguities; this will not impinge on any of what follows.
\end{note}

\begin{definition}
A morphism $f\colon\X\to\R$ is called a \emph{Morse function} if it has no nondegenerate critical points.
\end{definition}

This generalizes the usual definition of Morse function on a manifold to Deligne-Mumford stacks.  Just as for manifolds, Morse functions on a differentiable Deligne-Mumford stack are abundant: Theorem~\ref{MorseFunctionsAreDenseTheorem} will show that a generic morphism $\X\to\R$ is Morse in the sense that Morse functions form a dense open subset of the \emph{space} of all morphisms $\X\to\R$.

\begin{proposition}
$f\colon\X\to\R$ is Morse if and only if the composition $U\to\X\xrightarrow{f}\R$ is Morse for each \'etale $U\to\X$.  This is if and only if $f\circ\pi$ is Morse for some choice of \'etale atlas $\pi\colon X\to\X$.
\end{proposition}
\begin{proof}
The two claims are equivalent by Note~\ref{TangentSpaceNote}.  Any degenerate critical point of $f\circ\pi$ represents a degenerate critical point of $f$, and since $\pi$ is surjective any degenerate critical point of $f$ must be represented by a degenerate critical point of $f\circ\pi$.
\end{proof}

\begin{definition}\label{IndexCoindexDefinition}
Let $f\colon\X\to\R$ be Morse and let $c$ be a critical point of $f$.  Let $T_c\X=T_c\X_+\oplus T_c\X_-$ be an $\aut_c$-invariant splitting for which $H_{f,c}|_{T_c\X_+}$ is positive-definite and $H_{f,c}|_{T_c\X_-}$ is negative-definite.
\begin{enumerate}
\item The \emph{index} of $c$, denoted $\ind_c$, is the isomorphism class of $T_c\X_-$ as an $\aut_c$-representation.
\item The \emph{co-index} of $c$, denoted $\coind_c$, is the isomorphism class of $T_c\X_+$ as an $\aut_c$-representation.
\end{enumerate}
We refer to $\{(c,\aut_c,\ind_c,\coind_c)\mid c\mathrm{\ a\ critical\ point\ of\ }f\}$
as the \emph{critical point data for $f$}.
\end{definition}

The splitting required in Definition~\ref{IndexCoindexDefinition} can always be found.  It is not unique, but the isomorphism classes of $T_c\X_{\pm}$ as $\aut_c$-representations are uniquely determined.  

When $\X$ is a manifold the critical points of $f$ have trivial automorphism groups, so that the index and co-index are simply non-negative integers --- the dimensions of the relevant representations --- and moreover they determine one another since their sum is just the dimension of $\X$.  However, for general $\X$ the index and co-index may contain strictly more information than their dimensions, and they do not determine one another.  The most important piece of information contained in the index besides its dimension is singled out in the following definition.

\begin{definition}
A critical point $c$ of a Morse function $f\colon\X\to\R$ is called \emph{orientable} if the action of $\aut_c$ on $\ind_c$ is orientation-preserving.
\end{definition}

\subsection{The Morse Lemma.}\label{MorseLemmaSection}

\begin{theorem}[Morse Lemma]\label{MorseLemma}
Let $f\colon\X\to\R$ be a function and let $c\in\bar\X$ be a nondegenerate critical point of $f$.  Then there is an open subset $U_c\subset T_c\X$ and a linear orbifold-chart $[U_c/\aut_c]\to\X$ at $c$ for which
\[U_c\to[U_c/\aut_c]\to\X\xrightarrow{f}\R\]
is just $u\mapsto \bar f(c)+H_{f,c}(u,u)$.
\end{theorem}

Since the origin is the only critical point of $u\mapsto \bar{f}(c)+H_{f,c}(u,u)$  we immediately have the following:

\begin{corollary}\label{CriticalPointsIsolatedCorollary}
The critical points of a Morse function on $\X$ are isolated in $\bar\X$.
\end{corollary}

By taking an orbifold chart at $c$, the proof of the Morse Lemma reduces to the following $\aut_c$-equivariant form, which was proved by Lerman and Tolman in \cite{\LermanTolman}.  Its proof is based on the observation that Palais' proof of the Morse-Palais lemma (for Morse functions on Hilbert manifolds; see \cite{\Lang}) naturally extends to the equivariant setting.

\begin{lemma}[An equivariant  Morse lemma. {\cite{\LermanTolman}}]
Let $M$ be a manifold with $G$-action, $f\colon M\to\R$ a $G$-invariant function, and $m\in M$ a nondegenerate critical point of $f$ with stabilizer $G$.

There exists a $G$-equivariant diffeomorphism $\phi\colon U_0\to U_m$ from a neighbourhood of $0\in T_mM$ to a neighbourhood of $m\in M$ such that:
\begin{enumerate}
\item $\phi(0)=m$;
\item $f(\phi(v))-f(m)=H_{f,m}(v,v)$ for all $v\in U_0$;
\end{enumerate}
\end{lemma}

\subsection{Morse functions and the inertia stack.}\label{InertiaStackSection}

We will now show that a Morse function $f\colon\X\to\R$ on a differentiable Deligne-Mumford stack $\X$ induces a Morse function $f\circ\epsilon$ on the inertia stack $\Lambda\X$, and that the critical point data for $f\circ\epsilon$ can be read directly from that for $f$.

\begin{definition}
Recall, for example from \cite[4.4]{\AbramovichVistoliGraber}, that the \emph{inertia stack} of a stack $\X$ is the stack $\Lambda\X$ for which $\Lambda\X(U)$ is the following groupoid.  The objects are pairs $(u,\phi)$ for $u\in\X(U)$ and $\phi\in\aut_u$, and the morphisms $(u,\phi)\Rightarrow(v,\psi)$ are the $\lambda\colon u\Rightarrow v$ for which $\lambda\phi=\psi\lambda$.  If $\X$ is differentiable Deligne-Mumford then so is $\Lambda\X$, and there is a representable \emph{evaluation morphism} $\epsilon\colon\Lambda\X\to\X$ which on objects sends $(u,\phi)$ to $u$.
\end{definition}

The set $\overline{\Lambda\X}$ consists of pairs $(x,(g))$ where $x$ is a point of $\bar\X$ and $(g)$ is the conjugacy class of an element $g\in\aut_x$.  We will write $(x,(g))$ as $x^g$.  By taking an orbifold-chart at $x$ it is simple to see that $\aut_{x^g}=C_{\aut_x}(g)$ and that $T_{x^g}\Lambda\X=(T_x\X)^g$, where $-^g$ denotes fixed points of $g$.

\begin{theorem}\label{InertiaTheorem}
Let $f\colon\X\to\R$ be Morse.  Then $f\circ\epsilon\colon\Lambda\X\to\R$ is also Morse.  Moreover, the critical points of $f\circ\epsilon$ are precisely the $c^g$, where $c$ is a critical point of $f$ and $g\in\aut_c$.  Further,
\[\ind_{c^g}=(\ind_c)^g,\qquad \coind_{c^g}=(\coind_c)^g.\]
\end{theorem}

\begin{definition}\hfill
\begin{enumerate}
\item We say that a pair $(c,(g))$ is \emph{orientable} if $c^g$ is an orientable critical point of $\Lambda\X$.  Thus $(c,(g))$ is orientable if the action of $C_{\aut_c}(g)$ on $(\ind_c)^g$ orientation-preserving.
\item The \emph{age} or \emph{degree-shifting number} $\iota(c,(g))$ is the age $\iota_{T_c\X}(g)$, i.e.~the degree-shifting number associated to the component  of $\overline{\Lambda\X}$ containing $c^g$ \cite{\ChenRuan}.
\end{enumerate}
\end{definition}

\begin{proof}[Proof of Theorem~\ref{InertiaTheorem}.]
It is simple to check that for $x\in\bar\X$ and $g\in\aut_x$ we have $d_{x^g}(f\circ\epsilon)=d_{x}f|_{(T_x\X)^g}$, and that if $x$ is critical, then $H_{f\circ\epsilon,x^g}=H_{f,x}|_{(T_x\X)^g}$.  The result is now immediate from the following trivial lemma.
\end{proof}

\begin{lemma}
Let $V$ be a finite-dimensional real representation of a finite group $G$ and let $g\in G$.  Then:
\begin{enumerate}
\item Given a non-zero $G$-invariant linear map $d\colon V\to\R$, the restriction $d|\colon V^g\to\R$ is also non-zero.
\item Given a $G$-invariant nondegenerate symmetric bilinear form $H\colon V\times V\to\R$, the restriction $H|\colon V^g\times V^g\to\R$ is also nondegenerate.
\item Given a $G$-invariant splitting $V=V_+\oplus V_-$ there is a splitting $V^g=V_+^g\oplus V_-^g$.
\end{enumerate}
\end{lemma}

\section{Riemannian Metrics and Vector Fields.}\label{GeometrySection}

This section deals with Riemannian metrics and vector fields on differentiable Deligne-Mumford stacks.  The reason for covering these topics is that we wish to define the flow of the negative gradient field of a Morse function on a differentiable Deligne-Mumford stack.  This flow will be an elementary but crucial ingredient in proving the results of Section~\ref{ApplicationsSection} that relate the topology of $\bar\X$ to the critical points of a Morse function on $\X$.

In \S\ref{VFRMDefinitionSection} we define vector fields and Riemannian metrics on a differentiable Deligne-Mumford stack and we characterize them in terms of atlases.  We also state Theorem~\ref{MetricTheorem}, which tells us that any differentiable Deligne-Mumford stack admits a Riemannian metric.  In \S\ref{IntegralFlowSection} we define what it means for a morphism to integrate a vector field and we define flows of a vector field.  We then state Theorems \ref{IntegralTheorem} and \ref{FlowTheorem}, which tell us that integrals are unique (in an appropriate weak sense) and that any compactly-supported vector field has a flow.  All but the simplest proofs in \S\ref{VFRMDefinitionSection} and \S\ref{IntegralFlowSection} are deferred to \S\ref{ProofsSection}.

All of the definitions and results in \S\ref{VFRMDefinitionSection} and \S\ref{IntegralFlowSection} can be \emph{stated} for any stack that admits an \'etale atlas, and such stacks form a much larger class than the Deligne-Mumford stacks alone.  However, none of Theorems \ref{MetricTheorem}, \ref{IntegralTheorem} or \ref{FlowTheorem} remain true in this broader context.  This is the subject of \S\ref{CounterexamplesSection}, where we demonstrate the failure of these theorems by example.

\subsection{Riemannian metrics and vector fields.}\label{VFRMDefinitionSection}

An \'etale morphism of manifolds $f\colon U\to V$ induces isomorphisms of tangent spaces $df\colon T_uU\xrightarrow{\cong} T_{f(u)}V$ for each $u\in U$.  Therefore, given a Riemannian metric $\langle-,-\rangle$ on $V$, we obtain a metric $\langle-,-\rangle^f$ on $U$ by setting
\[\langle\alpha,\beta\rangle^f=\langle df(\alpha),df(\beta)\rangle.\]
Similarly, given a vector field $X$ on $V$, we obtain a vector field $f^\ast X$ on $U$ by setting
\[(f^\ast X)_u=(df)^{-1}X_{f(u)}.\]

\begin{definition}\label{RiemannianMetricDefinition}
A \emph{Riemannian metric} $\langle-,-\rangle$ on a differentiable Deligne-Mumford stack $\X$ is an assignment
\[(U\to\X)\mapsto\langle-,-\rangle_U\]
of a Riemannian metric $\langle-,-\rangle_U$ on $U$ to every \'etale morphism $U\to\X$, such that for every $2$-commutative diagram
\[\xymatrix{
V\ar[d]_f\ar[rr]_{}="1" & &\X\\
U\ar[rru]^{}="2"&&\ar@{=>}"1";"2"
}\]
we have ${\langle-,-\rangle_U}^f=\langle-,-\rangle_V$.  We will call a differentiable Deligne-Mumford stack equipped with a Riemannian metric a \emph{Riemannian differentiable Deligne-Mumford stack}.
\end{definition}

\begin{definition}\label{VectorFieldDefinition}
A \emph{vector field} $X$ on a differentiable Deligne-Mumford stack $\X$ is an assignment
\[(U\to\X)\mapsto X_U\]
of a vector field $X_U$ on $U$ to each \'etale morphism $U\to\X$, such that for every $2$-commutative diagram
\[\xymatrix{
V\ar[d]_f\ar[rr]_{}="1" & &\X\\
U\ar[rru]^{}="2"&&\ar@{=>}"1";"2"
}\]
we have $f^\ast X_U=X_V$.
\end{definition}

\begin{definition}\label{ManipulateVectorFieldsDefinition}
Let $\X$ be a Riemannian differentiable Deligne-Mumford stack.
\begin{enumerate}
\item Let $X$ be a vector field on $\X$ and $f\colon\X\to\R$ a morphism.  Then
\[X\cdot f\colon\X\to\R\]
denotes the morphism that when composed with an \'etale $\pi_U\colon U\to\X$ becomes $X_U\cdot(f\circ\pi_U)$.

\item Let $X$ and $Y$ be vector fields on $\X$. Then
\[\langle X,Y\rangle\colon\X\to\R\]
denotes the morphism that when composed with an \'etale morphism $U\to\X$ becomes $\langle X_U,Y_U\rangle_U\colon U\to\R$.

\item Let $f\colon\X\to\R$ be a morphism.  Then the \emph{gradient vector field of f}, denoted $\nabla f$, is defined by
\[(\nabla f)_U=\nabla(f\circ\pi_U)\]
for any \'etale morphism $\pi_U\colon U\to\X$, where the right hand side is formed using $\langle-,-\rangle_U$.  Note that $\langle\nabla f,X\rangle=X\cdot f$.
\end{enumerate}
\end{definition}

The gradient vector field gives us many examples of vector fields on any differentiable Deligne-Mumford stack that admits a Riemannian metric.  Metrics are provided by the following theorem.

\begin{theorem}\label{MetricTheorem}
Every differentiable Deligne-Mumford stack admits a Riemannian metric.
\end{theorem}

We have the following simple characterization of Riemannian metrics and vector fields on a differentiable Deligne-Mumford stack.  It implies in particular that Riemannian metrics and vector fields on a global quotient $[M/G]$ are in one-to-one correspondence with $G$-invariant Riemannian metrics and vector fields on $M$.

\begin{proposition}\label{RiemannianMetricVectorFieldAtlasProposition}
Let $\X$ be a differentiable Deligne-Mumford stack.  Let $A\to\X$ be an \'etale atlas and let $\pi_1,\pi_2\colon A\times_\X A\to A$ be the projections.  Then:
\begin{enumerate}
\item The assignment $\langle-,-\rangle\mapsto\langle-,-\rangle_A$ determines a one-to-one correspondence between Riemannian metrics $\langle-,-\rangle$ on $\X$ and Riemannian metrics $\langle-,-\rangle_A$ on $A$ that satisfy ${\langle-,-\rangle_A}^{\pi_1}={\langle-,-\rangle_A}^{\pi_2}$.
\item The assignment $X\mapsto X_A$ determines a one-to-one correspondence between vector fields $X$ on $\X$ and vector fields $X_A$ on $A$ that satisfy $\pi_1^\ast X_A=\pi_2^\ast X_A$.
\end{enumerate}
We call metrics that satisfy ${\langle-,-\rangle_A}^{\pi_1}={\langle-,-\rangle_A}^{\pi_2}$ and vector fields that satisfy $\pi_1^\ast X_A=\pi_2^\ast X_A$ \emph{invariant}.
\end{proposition}
\begin{proof}
We will prove the second result; the first is proved in exactly the same way.  Certainly, a vector field on $\X$ does induce an invariant vector field on $A$.  Conversely, suppose given an invariant vector-field $X_A$ on $A$ and let $U\to\X$ be \'etale.  Then in the diagram 
\[\xymatrix{
U\times_\X(A\times_\X A)\ar[r]\ar@<3 pt>[d]\ar@<-3 pt>[d] & A\times_\X A\ar@<3 pt>[d]\ar@<-3 pt>[d]\\
U\times_\X A\ar[r]\ar[d] & A\ar[d]_{}="2"\\
U\ar[r]^{}="1" & \X \ar@{=>}"2";"1"
}\]
we can use the first two horizontal maps, which are \'etale, to construct from $X_A$ a vector field on $U\times_\X A$ whose two pullbacks to $U\times_\X(A\times_\X A)$ coincide.  Since $U\times_\X(A\times_\X A)\rightrightarrows U\times_\X A$ is an \'etale groupoid representing $U$, this in turn induces a vector field $X_U$ on $U$.  The assignment $U\mapsto X_U$ clearly satisfies the required property.
\end{proof}

We now have the following proposition, which details three ways in which we can obtain new vector fields from old ones.  
\begin{proposition}\label{VectorFieldProposition}
Let $\X$ and $\Y$ be differentiable Deligne-Mumford stacks.
\begin{enumerate}
\item
Given vector fields $X,Y$ and functions $f,g$ on $\X$, there is a unique vector field $fX+gY$ on $\X$ such that for any $\pi_U\colon U\to\X$ \'etale,
\[(fX+gY)_U=(f\circ\pi_U)X_U+(g\circ\pi_U)Y_U.\]
\item
Given vector fields $X$ on $\X$ and $Y$ on $\Y$ there is a unique vector-field $X\oplus Y$ on $\X\times\Y$ such that
\[(X\oplus Y)_{U\times V}=X_U\oplus Y_V\]
for \'etale $U\to \X$, $Y\to\Y$.
\item
Let $\Y\to\X$ be an embedding and let $X$ be a vector field on $\X$.  Suppose that $X$ is \emph{tangent} to $\Y$ in the sense that for each \'etale $U\to\X$, $X_U$ is tangent to the submanifold $\Y\times_\X U\subset U$.  Then there is a unique vector field $X|_\Y$ on $\Y$ with $X_U|_{\Y\times_\X U}=(X|_\Y)_{\Y\times_\X U}$.  
\end{enumerate}
\end{proposition}

\subsection{Integrals and flows of a vector field.}\label{IntegralFlowSection}

Throughout what follows $\X$ and $\Y$ will denote differentiable Deligne-Mumford stacks, $X$ will be a vector field on $\X$, and $I\subset\R$ will be a possibly unbounded open interval.  We will write $\ddt$ for the vector field on $\Y\times I$ obtained by adding the zero vector field on $\Y$ and the vector field $\ddt$ on $I$ as in Proposition~\ref{VectorFieldProposition}.

\begin{definition}
A representable morphism $\Phi\colon\Y\times I\to\X$ \emph{integrates} $X$ if, for each \'etale $U\to\X$, the induced
\[\Phi_U\colon(\Y\times I)\times_\X U\to U\]
satisfies
\[\left(\bigddt_{(\Y\times I)\times_\X U}\right)\cdot\Phi_U=X_U\circ\Phi_U.\]
This condition holds for all \'etale $U\to\X$ if and only if it holds for a single \'etale atlas $A\to\X$.
\end{definition}

\begin{theorem}[Uniqueness of integrals.]\label{IntegralTheorem}
Let
\[\Phi,\Psi\colon\Y\times I\to\X\]
be representable morphisms that integrate $X$ and suppose given a $2$-morphism
\[\lambda\colon\Phi|_{\Y\times\{t_0\}}\Rightarrow\Psi|_{\Y\times\{t_0\}}\]
for some $t_0\in I$.  Then there is a unique $\Lambda\colon\Phi\Rightarrow\Psi$ for which $\Lambda|_{\Y\times\{t_0\}}=\lambda$.
\end{theorem}

\begin{definition}
A \emph{flow} $\Phi$ of $X$ is a representable morphism
\[\Phi\colon\X\times\R\to\X\]
that integrates $X$, together with a $2$-morphism $e_\Phi\colon\Phi|_{\X\times\{0\}}\Rightarrow\Id_\X$.
\end{definition}

The following is an immediate corollary of Theorem~\ref{IntegralTheorem}.

\begin{corollary}[Uniqueness of flows.]\label{FlowUniquenessCorollary}
Let $\Phi,\Psi\colon\X\times\R\to\X$ be flows of $X$.  Then there is a unique $2$-morphism $\lambda\colon\Phi\Rightarrow\Psi$ such that $\lambda|_{\X\times\{0\}}=e_\Psi^{-1}e_\Phi$.
\end{corollary}

\begin{proposition}\label{FlowActionProposition}
Let $\Phi$ be a flow of $X$.  Then there is a unique $2$-morphism $\mu_\Phi$
\[\xymatrix{
\X\times\R\times\R\ar[d]_{\Phi\times\Id_\R}\ar[r]^-{\Id_\X\times\alpha} & \X\times\R\ar[d]^{\Phi}_{}="1"\\
\X\times\R\ar[r]_\Phi^{}="2"&\X\ar@{=>}"2";"1"^{\mu_\Phi}
}\]
such that $\mu_\Phi|_{\X\times\{0\}\times\{0\}}=(\Phi|_{\X\times\{0\}})_\ast e_\Phi$.  Here $\alpha(s,t)=s+t$.  In particular, $\bar\Phi\colon\bar\X\times\R\to\bar\X$ is an action of $\R$ on $\bar\X$.
\end{proposition}

\begin{definition}
Given $U\to\X$ \'etale and $u\in U$, the question of whether $X_U(u)$ is zero or non-zero depends only on $[u\to\X]$.  The \emph{support of $X$} is defined to be \[\supp X=\cl\{x\in\bar\X\mid x=[u\to\X],\, X_U(u)\neq 0\}.\]
\end{definition}

\begin{theorem}[Existence of flows.]\label{FlowTheorem}
Suppose that $\supp X$ is compact.  Then there is a flow $\Phi\colon\X\times\R\to\X$ of $X$.
\end{theorem}

\begin{proposition}[Restriction of flows.]\label{RestrictedFlowProposition}
Let $i\colon\Y\to\X$ be an embedding and suppose that $X$ is tangent to $\Y$.  Then a flow of $X$ on $\X$ can be restricted to a flow of $X|_\Y$ on $\Y$.  That is, for any flow $\Phi\colon\X\times\R\to\X$ of $X$, there is a flow $\Psi\colon\Y\times\R\to\Y$ of $X|_\Y$ and a $2$-commutative diagram
\[\xymatrix{
\Y\times\R\ar[r]^\Psi\ar[d]_{i\times\Id_\R} & \Y\ar[d]^i_{}="1"\\
\X\times\R\ar[r]_{\Phi}^{}="2" & \X\ar@{=>}_{\varepsilon}"1";"2"
}\]
with $\varepsilon|_{\Y\times\{0\}}=i_\ast e_\Psi \circ i^\ast e_\Phi^{-1}$.
\end{proposition}

\begin{lemma}\label{UnderlyingFlowLemma}
Let $X$ be a vector field on $\X$ with flow $\Phi$, and let $f\colon\X\to\R$ be a morphism.  Write $\varphi_t\colon\bar\X\to\bar\X$ for the action underlying $\Phi$.  Then for any $x\in\bar\X$ the map $t\mapsto\bar f\circ\varphi_t(x)$ is smooth with derivative $t\mapsto\overline{X\cdot f}(\varphi_t(x))$.
\end{lemma}

\subsection{Counterexamples for non-Deligne-Mumford stacks.}\label{CounterexamplesSection}

With the exception of Theorem~\ref{MetricTheorem}, the definitions and results of \S\ref{VFRMDefinitionSection} apply to any differentiable stack that admits an \'etale atlas.  We can therefore ask whether Theorems \ref{MetricTheorem}, \ref{IntegralTheorem} and \ref{FlowTheorem} hold for these more general stacks.  The answer in each case is `no', even if one restricts to stacks with an \'etale atlas and finite inertia groups, as we shall now show.

\begin{example}\label{ExampleOne}
In this example we will show that Theorem~\ref{MetricTheorem} can fail for stacks which admit \'etale atlases and have finite inertia groups.

Define $\mathfrak{C}=[C_1\rightrightarrows C_0]$, where $C_1\rightrightarrows C_0$ is the groupoid with objects $\R$ and with, for each $n\in\mathbb{N}$, a single morphism from each $t\in(\frac{-1}{2^n},\frac{-1}{2^{n+1}})$ to $2+2^{n+1}t\in(0,1)$, and with no nontrivial morphisms besides the ones these generate.  Thus $C_0=\R$, while $C_1$ is the disjoint union of one copy of $\R$ with countably many copies of $(0,1)$.  The groupoid $C_1\rightrightarrows C_0$ is \'etale and has trivial inertia groups but is not proper.

Suppose that $\mathfrak{C}$ admits a Riemannian metric and consider the corresponding invariant metric on $C_0$.  Without loss let the length of $\partial/\partial x\in T_0 C_0$ be $1$, so that $\partial/\partial x\in T_tC_0$ has length at most $2$ for all $t$ in some neighbourhood of $0$; this neighbourhood contains $(\frac{-1}{2^n},\frac{-1}{2^{n+1}})$ for all $n$ large enough.  For each $n$, $C_1$ contains a copy of $(0,1)$ with $s\colon(0,1)\to\R$ given by $r\mapsto r$ and $t\colon(0,1)\to\R$ given by $r\mapsto(r-2)/2^{n+1}$.  Invariance of the metric then means that on $(0,1)$ $\partial/\partial x\in T_tC_0$ must have length at most $1/2^n$ for each $n$ large enough.  Thus $\partial/\partial x$ must have length zero at all points on $(0,1)$, which is a contradiction.
\end{example}

\begin{example}\label{ExampleTwo}
In this example we will show that Theorem~\ref{IntegralTheorem} can fail for vector fields on stacks which admit \'etale atlases and have finite inertia groups.

Let $\mathfrak{A}$ be the stack $[A_1\rightrightarrows A_0]$, where $A_1\rightrightarrows A_0$ is the \'etale groupoid whose objects consist of two copies of $\R$ that we denote by $\R_0$ and $\R_1$, and with a single morphism from $t\in\R_0$ to $t\in\R_1$ for each $t\in(-\infty,0)$, and no further nontrivial morphisms besides the ones these generate.  Thus $A_0=\R_0\sqcup\R_1$ and $A_1$ is a disjoint union of copies of $\R$ and $(-\infty,0)$.  Then $A_1\rightrightarrows A_0$ is an \'etale groupoid with trivial inertia groups but is not proper.

The underlying space $\bar{\mathfrak{A}}$ consists of two copies of $\R$ with the subsets $(-\infty,0)$ identified, and so is not Hausdorff.  In particular, for any $\delta>0$ the two morphisms $\{-\delta\}\hookrightarrow\R_i\to\mathfrak{A}$ are $2$-isomorphic while the two $\{\delta\}\hookrightarrow\R_i\to\mathfrak{A}$ are not.

Now consider the vector field $A$ on $\mathfrak{A}$ corresponding to the invariant vector field $\partial/\partial x$ on $\R_0\sqcup\R_1$.  The two morphisms $\R=\R_i\to\mathfrak{A}$ then integrate $A$ and are $2$-isomorphic when restricted to $\{-1\}$, but are not themselves $2$-isomorphic.  Thus Theorem~\ref{IntegralTheorem} fails for $\mathfrak{A}$ and $A$.
\end{example}

\begin{example}\label{ExampleThree}
In this example we will show that Theorem~\ref{FlowTheorem} can fail for vector fields on stacks which admit \'etale atlases and have finite inertia groups.  

Let $\mathfrak{A}$ be the stack defined in the last example, and let $B$ be the vector field on $\mathfrak{A}$ corresponding to the invariant vector field $\phi\cdot\partial/\partial x$ on $\R_0\sqcup\R_1$.  Here $\phi\colon\R_0\sqcup\R_1\to\R$ is the composition of the componentwise identity map $\R_0\sqcup\R_1\to\R$ with a function $\R\to\R$ that has value $1$ on $[-2,2]$ and that has compact support.

$B$ has compact support, so let us suppose that Theorem~\ref{FlowTheorem} holds for $\mathfrak{A}$ and $B$, giving us $\Phi\colon\mathfrak{A}\times\R\to\mathfrak{A}$ and $e_\Phi\colon\Phi_0\Rightarrow\Id$.  It is possible to show that there must be some $\epsilon>0$ such that for each $i$ the composition
\[(-\epsilon,\epsilon)\xrightarrow{\{0\}\times\mathrm{inc}}\R_i\times\R\to\mathfrak{A}\times\R\xrightarrow{\Phi}\mathfrak{A}\]
is $2$-isomorphic to
\[(-\epsilon,\epsilon)\to\R_i\to\mathfrak{A}.\]
However, the first pair of morphisms are $2$-isomorphic, while the second pair of morphisms are not; this is a contradiction.
\end{example}

\subsection{Proofs.}\label{ProofsSection}

We will now give proofs of the results of \S\ref{VFRMDefinitionSection} and \S\ref{IntegralFlowSection}.  We begin with the proofs of Proposition~\ref{VectorFieldProposition} and Theorem~\ref{MetricTheorem} from \S\ref{VFRMDefinitionSection}.

\begin{proof}[Proof of Proposition~\ref{VectorFieldProposition}.]
The first part is immediate from the fact that, given $h\colon U\to V$ \'etale and vector fields $X,Y$ and functions $f,g$ on $V$, we have $h^\ast(fX+gY)=(f\circ h)h^\ast X+(g\circ h)h^\ast Y$.

Now we prove the second part.  Using Proposition~\ref{RiemannianMetricVectorFieldAtlasProposition} and choosing \'etale atlases $A\to\X$, $B\to\Y$, we can define $X\oplus Y$ to be the vector field on $\X\times\Y$ for which
\[(X\oplus Y)_{A\times B}=X_A\oplus Y_B.\]
We must now verify that $(X\oplus Y)_{U\times V}=X_U\oplus Y_V$ for any \'etale maps $U\to\X$ and $V\to\Y$.  But in the diagram
\[\xymatrix{
(U\times V)\times_{\X\times\Y}(A\times B)\ar[r]^-{\pi_2}\ar[d]_{\pi_1} & A\times B\ar[d]_{}="1"\\
U\times V\ar[r]^{}="2" & \X\times\Y\ar@{=>}"1";"2"
}\]
we have $(U\times V)\times_{\X\times\Y}(A\times B)=(U\times_\X A)\times(V\times_\Y B)$, $\pi_1=\pi_1\times\pi_1$, and $\pi_2=\pi_2\times\pi_2$.  Therefore $\pi_1^\ast(X_U\oplus Y_V)=X_{U\times_\X A}\oplus Y_{V\times_\Y B}=\pi_2^\ast(X_A\oplus Y_B)=\pi_2^\ast(X\oplus Y)_{A\times B}=\pi_1^\ast(X\oplus Y)_{U\times V}$ so that, since $\pi_1$ is surjective and \'etale, we must have $X_U\oplus Y_V=(X\oplus Y)_{U\times V}$.

The third part is proved in a similar way.  Let $B\to\Y$ be the induced \'etale atlas $A\times_\X\Y\to\Y$, which is a submanifold of $A$.  Define $X|_\Y$ to be the vector field with $(X|_\Y)_B=X_A|_B$.  Now we must verify that $(X|_\Y)_{U\times_\X\Y}=X_U|_{U\times_\X\Y}$ for any \'etale $U\to\X$.  Consider the diagrams
\[\xymatrix{
U\times_\X A\ar[r]^-{\pi_1}\ar[d]_{\pi_2} & U\ar[d]_{}="1"\\
A\ar[r]^{}="2" & \X,\ar@{=>}"1";"2"}\qquad
\xymatrix{
(U\times_\X\Y)\times_\Y B\ar[r]^-{\pi_1}\ar[d]_{\pi_2} & U\times_\X\Y\ar[d]_{}="1"\\
B\ar[r]^{}="2" & \Y,\ar@{=>}"1";"2"
}\]
where the second diagram is obtained from the first by pulling back along $\Y\to\X$.  We therefore have $\pi_1^\ast(X|_\Y)_{U\times_\X\Y}=\pi_2^\ast(X|_\Y)_B=\pi_2^\ast(X_A|_B)=(\pi_2^\ast X_A)|_{(U\times_\X\Y)\times_\Y B}=(\pi_1^\ast X_U)|_{(U\times_\X\Y)\times_\Y B}=\pi_1^\ast (X_U|_{U\times_\X\Y})$.  Since $\pi_1$ is an \'etale surjection, we must have $(X|_\Y)_{U\times_\X\Y}=X_U|_{U\times_\X\Y}$ as required.
\end{proof}

\begin{proof}[Proof of Theorem~\ref{MetricTheorem}.]
By Proposition~\ref{OrbifoldChartsProposition} we may cover $\bar\X$ with open sets of the form $\overline{[M/G]}$ for $[M/G]\to\X$ an orbifold chart.  Applying Theorem~\ref{PartitionsOfUnityTheorem}, we may find $\phi_i\colon\X\to\R$, for $i=1,2,\ldots$, such that the the $\bar\phi_i$ are a partition of unity on $\bar\X$ and each $\supp\phi_i$ is contained in $\overline{[M_i/G_i]}$ for some orbifold-chart $[M_i/G_i]\to\X$.  We write $\iota_i$ for the composition $M_i\to[M_i/G_i]\to\X$.  Note that by averaging we may find for each $i$ a $G_i$-invariant metric $\langle-,-\rangle_i$ on $M_i$.

With this data we will now construct the required assignment $(U\to\X)\mapsto\langle-,-\rangle_U$ satisfying the conditions of Definition~\ref{RiemannianMetricDefinition}.

Let $U\to\X$ be \'etale and let $u\in U$.  Then there is an open neighbourhood $V$ of $u$ such that $(\supp\phi_i)_V$ is non-empty for only finitely many $i$ and, reducing $V$ if necessary, $(\supp\phi_i)_V$ is non-empty only if it contains $u$.  For each such $i$ the map $V\times_\X M_i\to V$ is \'etale and has image containing $u$.  We may therefore find a neighbourhood $W_u$ of $u$ and $2$-commutative diagrams
\[\xymatrix{W_u\ar[r]\ar[d]_{\lambda_i} & U\ar[d]_{}="1"\\
M_i\ar[r]^{}="2"_{\iota_i}&\X\ar@{=>}"1";"2"_{\psi_i}}\]
for each $i$ such that $(\supp\phi_i)_V\neq\emptyset$.  Now consider the Riemannian metric $\langle-,-\rangle_{W_u}$ on $W_u$ defined by
\begin{equation}\label{MetricEquation}\langle\alpha,\beta\rangle_{W_u}=\sum\phi_i(\iota_i(\lambda_i(w))){\langle\alpha,\beta\rangle_i}^{\lambda_i}\end{equation}
for $\alpha,\beta\in T_wW_u$.  This is certainly a smooth section of $S^2T^\ast W_u$, and it is positive since $\sum\bar\phi_i=1$.

Now let $\langle-,-\rangle_u$ denote the metric on $T_uU$ induced by $\langle-,-\rangle_{W_u}$.  We claim that this metric is independent of the choices made.  For suppose given a second set of data $W'_u$, $\lambda_i'$, $\psi_i'$.  We may clearly asume that $W_u=W_u'$ and that this neighbourhood of $u$ is connected.  Then for each $i$ there is $g_i\in G_i$ such that $\lambda_i'=g_i\lambda_i$, and so ${\langle-,-\rangle_i}^{\lambda'_i}={\langle-,-\rangle_i}^{\lambda_i}$ since $\langle-,-\rangle_i$ is $G_i$-invariant.  Also $\iota_i\circ\lambda_i'$ and $\iota_i\circ\lambda_i$ are $2$-isomorphic, so that $\phi_i\circ\iota_i\circ\lambda_i=\phi_i\circ\iota_i\circ\lambda_i'$.  The claim is now immediate from \eqref{MetricEquation}.

Since $W_u$ serves as $W_v$ for any $v\in W_u$, the $\langle-,-\rangle_u$ combine to give a Riemannian metric $\langle-,-\rangle_U$ on $U$.  We must now check  that, given a $2$-commutative diagram
\[\xymatrix{
V\ar[d]_f\ar[rr]_{}="1"^{\pi_V} & &\X\\
U\ar[rru]^{}="2"_{\pi_U}&&\ar@{=>}_{\psi}"1";"2"
}\]
with $f$ \'etale, we have ${\langle-,-\rangle_U}^f=\langle-,-\rangle_V$.  Given $v\in V$, take a neighbourhood $W_v$ as above, chosen small enough that $f|\colon W_v\to U$ is an open embedding.  We may therefore regard $f|\colon W_v\to U$ as the inclusion of $f(W_v)$.  From the diagrams
\[\xymatrix{W_v\ar[r]\ar[d]_{\lambda_i} & V\ar[d]_{}="1"^{\pi_V}\\
M_i\ar[r]^{}="2"_{\iota_i}&\X\ar@{=>}"1";"2"_{\psi_i}}\]
we obtain diagrams
\[\xymatrix{f(W_v)\ar[r]\ar[d]_{\lambda_if|^{-1}} & U\ar[d]_{}="1"^{\pi_U}\\
M_i\ar[r]^{}="2"_{\iota_i}&\X,\ar@{=>}"1";"2"_{\psi_i\psi^{-1}}}\]
so that again looking at the definition \eqref{MetricEquation}, $f|\colon W_v\to f(W_v)$ identifies $\langle-,-\rangle_U$ and $\langle-,-\rangle_V$ over these subsets.  The result follows.
\end{proof}

Now we deal with the proof of Theorem~\ref{IntegralTheorem}, which is based on the following lemma.

\begin{lemma}\label{PartialIntegralLemma}
Let $Y$ be a manifold and let
\[\Phi,\Psi\colon Y\times I\to\X\]
be two morphisms integrating $\X$.  Let $t\in I$ and let $Y_1\subset Y$ be an open subset with $\cl Y_1$ compact.  Then there is an open interval $J\subset I$ with $t\in J$ such that any $2$-morphism
\[\lambda\colon\Phi|_{Y_1\times\{s\}}\Rightarrow\Psi|_{Y_1\times\{s\}}\]
with $s\in J$ extends to a $2$-morphism $\Lambda\colon\Phi|_{Y_1\times J}\Rightarrow\Psi|_{Y_1\times J}$.
\end{lemma}

\begin{proof}[Proof of Theorem \ref{IntegralTheorem}.]
We will first prove the theorem when $\Y$ is a manifold $Y$.  Consider $2$-morphisms $\Lambda_{Y_1,J}\colon\Phi|_{Y_1\times J}\Rightarrow\Psi|_{Y_1\times J}$ such that $\Lambda_{Y_1,J}|_{Y_1\times\{t_0\}}=\lambda|_{Y_1}$, where $Y_1$ is an open subset of $Y$ and $J\subset I$ is an open interval containing $t_0$.  Such a $\Lambda_{Y_1,J}$ is unique, if it exists.  Our aim is to show that $\Lambda_{Y,I}$ exists.

Fix $Y_1\subset Y$ open with $\cl Y_1$ compact.  By applying Lemma~\ref{PartialIntegralLemma} with $s=t=t_0$ one can find some $J$ for which $\Lambda_{Y_1,J}$ exists.  Moreover, if one can find $\Lambda_{Y_1,J}$ and $\Lambda_{Y_1,J'}$ then, using the uniqueness of $\Lambda_{Y_1,J\cap J'}$, these can be glued together to produce $\Lambda_{Y_1,J\cup J'}$.  Therefore let $K\subset I$ be the largest open interval for which $\Lambda_{Y_1,K}$ exists.

We claim that $K=I$.  If not, then let $(t_i)\i$ be a sequence in $K$ converging to $t\in I-K$.  There is an open interval $J\subset I$ containing $t$ and satisfying the conclusions of Lemma~\ref{PartialIntegralLemma}.  Since $J$ contains $t$, it contains some $t_i$, and so applying the conclusion of Lemma~\ref{PartialIntegralLemma} with $s=t_i$ and $\lambda=\Lambda_{Y_1,K}|_{Y_1\times\{t_i\}}$, there is $M\colon\Phi|_{Y_1\times J}\Rightarrow\Psi|_{Y_1\times J}$ with $M|_{Y_1\times\{t_i\}}=\Lambda_{Y_1\times K}|_{Y\times\{t_i\}}$.  It follows that $M|_{Y_1\times J\cap K}=\Lambda_{Y_1,K}|_{Y_1\times J\cap K}$, and therefore $M$ and $\Lambda_{Y_1,K}$ can be glued to produce $\Lambda_{Y_1,K\cup J}$, contradicting the assumption.

We have established that $\Lambda_{Y_1,I}$ exists for each $Y_1$ with $\cl Y_1$ compact.    We can write $Y=\bigcup_{i\in\mathbb{N}}Y_i$, where each $Y_i\subset Y$ is open with $\cl Y_i$ compact, and we can find $\Lambda_{Y_i,I}$ for each $i$.  Since $\Lambda_{Y_i,I}|_{Y_i\cap Y_j}=\Lambda_{Y_j,I}|_{Y_i\cap Y_j}$, the $\Lambda_{Y_i,I}$ can be glued to produce the required $\Lambda$.  This completes the proof in the case $\Y=Y$.

We now turn to the general case.  Let $y\colon Y\to\Y$ be an \'etale surjection.  Then by the result for manifolds, there is $\Lambda_Y\colon \Phi\circ(y\times\Id_I)\Rightarrow\Psi\circ(y\times\Id_I)$ such that $\Lambda_Y|_{Y\times\{t_0\}}=y^\ast\lambda$.  Consider the pullback diagram
\[\xymatrix{
Y\times_\Y Y\ar[r]^-s\ar[d]_t & Y\ar[d]_{}="1"\\
Y\ar[r]_y^{}="2" & \Y.\ar@{=>}"1";"2"_{\mu}
}\]
In order to show that $\Lambda_Y$ descends to the required $\Lambda$, it suffices to show that
\begin{gather*}
\Psi_\ast(\mu\times\Id_{\Id_I})\circ(s\circ y\times\Id_I)^\ast\Lambda_Y,\\
(t\circ y\times\Id_I)^\ast\Lambda_Y\circ\Phi_\ast(\mu\times\Id_{\Id_I}),
\end{gather*}
which are $2$-morphisms $\Phi\circ(y\circ s\times\Id_I)\Rightarrow\Psi\circ(y\circ t\times\Id_I)$, coincide.  When restricted to $Y\times_\Y Y\times\{t_0\}$ these become $({\Psi|_{\Y\times\{t_0\}}}_\ast\mu)\circ(s\circ y^\ast\lambda)$, $(t\circ y^\ast\lambda)\circ{\Phi|_{\Y\times\{t_0\}}}_\ast\mu$, which do indeed coincide, and therefore $\Lambda_Y$ descends to $\Lambda$ as required.
\end{proof}

We must now prove Lemma~\ref{LocalIntegralLemma}.  In order to do so we state and prove the following two supporting lemmas.  

\begin{lemma}\label{LocalIntegralLemma}
Let $Y$ be a manifold, $\Phi\colon Y\times I\to\X$ a morphism that integrates $X$, and $U\to\X$ an \'etale surjection.  Then for each open $Y_1\subset Y$ with $\cl{Y_1}$ compact and each $t\in I$ there is an interval $J\subset I$ containing $t$, an \'etale surjection $\tilde Y_1\to Y$, and a $2$-commutative square
\[\xymatrix{
{\tilde Y}_1\times J\ar[r]\ar[d] & U\ar[d]_{}="1"\\
Y\times I\ar[r]_\Phi^{}="2" & \X\ar@{=>}"1";"2"
}\]
in which $\tilde Y_1\times J\to U$ integrates $X_U$.
\end{lemma}
\begin{proof}
Consider the pullback diagram:
\[\xymatrix{
U\times_\X(Y\times I)\ar[d]\ar[r]^-{{\Phi_U}}  & U\ar[d]_{}="1"\\
Y\times I\ar[r]_\Phi^{}="2" & \X\ar@{=>}"1";"2"
}\]
For each $y\in\cl Y_1$ choose a $\tilde y\in U\times_\X(Y\times I)$ lying over $(y,t)\in Y\times I$.  We can find an open neighbourhood $V_y$ of $y$ and $\delta_y>0$ such that $V_y\times(t-\delta_y,t+\delta_y)$ lifts to a neighbourhood of $\tilde y$; on such a neighbourhood $\Phi_U\colon V_y\times(t-\delta_y,t+\delta_y)\to U$ integrates $X_U$.  Since the $V_y$ cover $\cl Y_1$ and $\cl Y_1$ is compact, we can choose $y_1,\ldots,y_n\in \cl Y_1$ such that $\cl Y_1\subset \bigcup V_{y_i}$.  Set $\delta=\mathrm{min}(\delta_{y_i})$ and $\tilde Y_1=\bigsqcup V_{y_i}\cap Y_1$, so that $\tilde Y_1\to Y_1$ is an \'etale surjection, and the previous diagram gives us the required square
\[\xymatrix{
\tilde Y_1\times I\ar[d]\ar[r]&U\times_\X(Y\times I)\ar[r]^-{\Phi_U} & U\ar[d]_{}="1"\\
Y\times I\ar[rr]_\Phi^(0.75){}="2"& & \X.\ar@{=>}"1";"2"
}\]
\end{proof}

\begin{lemma}\label{PointIntegralLemma}
Let $U\to\X$ be \'etale, and suppose given $\gamma,\delta\colon I\to U$ that integrate $X_U$, together with $e\in U\times_\X U$ such that $s(e)=\gamma(t_0)$, $t(e)=\delta(t_0)$.  Then there is a unique $\varepsilon\colon I\to U\times_\X U$ such that $\varepsilon(t_0)=e$, $s\circ\varepsilon=\gamma$, $t\circ\varepsilon=\delta$.
\end{lemma}
\begin{proof}
We must show that the integral curve $\varepsilon$ of $X_{U\times_\X U}$ with $\varepsilon(t_0)=e$ can be defined on the interval $I$.  Let $J\subset I$ be the largest open interval containing $t_0$ on which $\varepsilon$ is defined.  If $J\neq I$, let $t_i$ be a sequence in $J$ that converges to some $t\in I-J$.  Choose closed discs $D_\gamma,D_\delta\subset U$ around $\gamma(t)$, $\delta(t)$ respectively, and without loss assume $\gamma(t_i)\in D_\gamma$, $\delta(t_i)\in D_\delta$ for all $i$.  Then $(s\times t)^{-1}D_\gamma\times D_\delta$ is a compact subset of $U\times_\X U$ containing the points $\varepsilon(t_i)$ for all $i$, and therefore there is a subsequence $t_{i_j}$ of $t_i$ and an $f\in(s\times t)^{-1}D_\gamma\times D_\delta$ for which $\varepsilon(t_{i_j})$ converges to $f$.  Note that $s(f)=\gamma(t)$ and $t(f)=\delta(t)$.

Choose an open neighbourhood $N$ of $f$ for which $s|_N$ and $t|_N$ are diffeomorphisms onto their images.  There is an open interval $K\subset I$ containing $t$ such that $\gamma|_K$ has image in $s(N)$.  Lift $\gamma|_K$ to a curve $g$ in $N$.  Then $g$ integrates $X_{U\times_\X U}$, and note that we must have $\varepsilon(t_{i_j})=g(t_{i_j})$ for all $j$ large enough.  This shows that $\varepsilon$ can be defined on an interval containing $t$, in contradiction with the initial assumption.
\end{proof}

\begin{proof}[Proof of Lemma~\ref{PartialIntegralLemma}.]
By applying Lemma~\ref{LocalIntegralLemma} for $Y_1$, $t$, and each of $\Phi$ and $\Psi$, we can find an open interval $J$ containing $t$, an \'etale surjection $\tilde Y_1\to Y_1$, and $2$-commutative diagrams
\[\xymatrix{
\tilde Y_1\times J \ar[r]^{\tilde\Phi}\ar[d] & U\ar[d]_{}="1" && \tilde Y_1\times J\ar[r]^{\tilde\Psi}\ar[d] & U\ar[d]_{}="3"\\
Y\times I\ar[r]_\Phi^{}="2" & \X, && Y\times I\ar[r]_\Psi^{}="4"& \X. \ar@{=>}"1";"2"\ar@{=>}"3";"4"
}\]

Now suppose given $\lambda\colon\Phi|_{Y_1\times\{s\}}\Rightarrow\Psi|_{Y_1\times\{s\}}$ as in the statement.  This $2$-morphism, together with the two diagrams above, induces a map $l\colon\tilde Y_1\times\{s\}\to U\times_\X U$ with $s\circ l=\tilde\Phi|_{\tilde Y_1\times\{s\}}$, $t\circ l=\tilde\Psi|_{\tilde Y_1\times\{s\}}$, and whose compositions with the two maps $\tilde Y_1\times_Y\tilde Y_1\to\tilde Y_1$ coincide.

By Lemma~\ref{PointIntegralLemma}, for any $y\in \tilde Y_1$ the integral curve $\varepsilon$ of $X_{U\times_\X U}$ with $\varepsilon(s)=l(y)$ can be defined on the interval $J$.  We can therefore extend $l$ to a map $L\colon\tilde Y_1\times J\to U\times_\X U$ integrating $X_{U\times_\X U}$.  Since $L$ integrates $X_{U\times_\X U}$ it follows that $s\circ L=\tilde\Phi$ and $t\circ L=\tilde\Psi$, and that the compositions of $L$ with the two maps $\tilde Y_1\times_Y\tilde Y_1\times J\to\tilde Y_1\times J$ coincide.  

The map $L$ and its stated properties lead immediately to the required $2$-morphism $\Lambda$.
\end{proof}

This concludes the material related to Theorem~\ref{IntegralTheorem}.  Now we go on to deal with the proofs of Proposition~\ref{FlowActionProposition}, Theorem~\ref{FlowTheorem} and Proposition~\ref{RestrictedFlowProposition}.

\begin{proof}[Proof of Proposition~\ref{FlowActionProposition}.]
First consider the two compositions when restricted to $\X\times\{0\}\times\{0\}$.  These are $\Phi|_{\X\times\{0\}}\circ\Phi|_{\X\times\{0\}}$ and $\Phi|_{\X\times\{0\}}$, and so we have the $2$-morphism $(\Phi|_{\X\times\{0\}})_\ast e_\Phi$ between them.  Now consider the two compositions when restricted to $\X\times\R\times\{0\}$.  These both integrate $\X$, so by Theorem~\ref{IntegralTheorem} there is a $2$-morphism $M$ between them that restricts to $(\Phi|_{\X\times\{0\}})_\ast e_\Phi$ on $\X\times\{0\}\times\{0\}$.  Finally, the two compositions, regarded now as morphisms $(\X\times\R)\times\R\to\X$, both integrate $X$ and are $2$-isomorphic when restricted to $\X\times\R\times\{0\}$, so that by Theorem~\ref{IntegralTheorem} we obtain the required $\mu_\Phi$.
\end{proof}

\begin{proof}[Proof of Theorem~\ref{FlowTheorem}.]
Let $U\to\X$ be an \'etale atlas and let $X_U$ be the vector field on $U$ induced by $X$.  For each $u\in U$ there is an open neighbourhood $V_u$ of $u$ and $\delta_u >0$ such that the flow
\[\Phi_u\colon V_u\times(-\delta_u,\delta_u)\to U\]
of $X_U$ is defined, which means that
\begin{gather*}
\Phi_u(v,0)=v,\\
\bigddt\Phi_u=X_U\circ\Phi_u.
\end{gather*}
If $u\not\in\supp X_U$, then without loss let $V_u\subset U-\supp X_U$ and $\delta_u=1$.  Choose open neighbourhoods $W_u\subset V_u$ of $u$ for which $\cl W_u$ is compact and still contained in $V_u$.  Then for $u_1,u_2\in U$, $W_{u_1}\times_\X W_{u_2}\subset U\times_\X U$ is contained within the compact $(s\times t)^{-1}\cl W_{u_1}\times \cl W_{u_2}$, and so there is $\delta_{u_1 u_2}>0$ such that the flow
\[\Phi_{u_1u_2}\colon W_{u_1}\times_\X W_{u_2}\times(-\delta_{u_1u_2},\delta_{u_1u_2})\to U\times_\X U\]
of $X_{U\times_\X U}$ is defined.  If $u_1\not\in\supp X_U$ or $u_2\not\in\supp X_U$ then without loss let $\delta_{u_1u_2}=1$.
Since $\supp X$ is compact we can choose countably many $u_i\in U$, of which only finitely many lie in $\supp X_U$, such that any point in $\X$ is $2$-isomorphic to a point in one of the $W_{u_i}$.  Set
\begin{gather*}
U'=\bigsqcup_i W_{u_i},\\
\delta=\min\{\delta_{u_i},\delta_{u_iu_j}\}.
\end{gather*}
Note that $\delta>0$ since only finitely many of the $\delta_{u_i}$, $\delta_{u_iu_j}$ are not equal to $1$.

We now have two groupoids, $[U'\times_\X U'\rightrightarrows U']$ and $[U\times_\X U\rightrightarrows U]$, that represent $\X$, and a groupoid-morphism
\[(i_0,i_1)\colon[U'\times_\X U'\rightrightarrows U']\to[U\times_\X U\rightrightarrows U],\]
obtained by componentwise inclusion, that represents $\Id_\X$.  Moreover, the $\Phi_{u_i}$ and $\Phi_{u_iu_j}$ define a faithful groupoid-morphism
\[(\Phi_0,\Phi_1)\colon[U'\times_\X U'\times(-\delta,\delta)\rightrightarrows U'\times(-\delta,\delta)]\to[U\times_\X U\rightrightarrows U]\]
in which $\Phi_0$ integrates $X_U$ and restricts to $i_0$ on $U'\times\{0\}$, and in which $\Phi_1$ integrates $X_{U\times_\X U}$ and restricts to $i_1$ on $U'\times_\X U'\times\{0\}$.  (The last two claims are immediate by construction; that $(\Phi_0,\Phi_1)$ is a groupoid-morphism then follows from the same fact for $(i_0,i_1)$ and the fact that the $\Phi_i$ integrate the stated vector fields.  We now show that $(\Phi_0,\Phi_1)$ is faithful.  Suppose given points $(w_1,t),(w_2,t)\in U'\times(-\delta,\delta)$ and $(v,t),(v',t)\in U'\times_\X U'\times(-\delta,\delta)$ such that $s(v,t)=s(v',t)=(w_1,t)$ and $t(v,t)=t(v',t)=(w_2,t)$.  Suppose that $\Phi_1(v,t)=\Phi_1(v',t)$.  Then, since $\Phi_1$ is the flow of $X_{U\times_\X U}$, it follows that $v=v'$, so that $(v,t)=(v',t)$ and $(\Phi_0,\Phi_1)$ is faithful as claimed.)

It follows from the last paragraph that  there is a morphism of stacks $\Phi\colon\X\times (-\delta,\delta)\to\X$ and a $2$-commutative diagram
\[\xymatrix{
U'\times(-\delta,\delta)\ar[r]^-{\Phi_0}\ar[d] & U\ar[d]_{}="1"\\
\X\times(-\delta,\delta)\ar[r]_-\Phi^{}="2" & \X \ar@{=>}"1";"2"
}\]
that induces $(\Phi_0,\Phi_1)$.  Since $(\Phi_0,\Phi_1)$ is faithful, $\Phi$ is representable.  Since $\Phi_0$ integrates $X_U$ it follows that $\Phi$ integrates $X$.  Since $(\Phi_0,\Phi_1)$ restricts to $(i_0,i_1)$ on $[U'\times_\X U'\times\{0\}\rightrightarrows U'\times\{0\}]$ it follows that there is a $2$-morphism $e_\Phi\colon\Phi|_{\X\times\{0\}}\to \Id_\X$. 

We will now use $\Phi\colon\X\times(-\delta,\delta)\to\X$ to construct the required $\Phi\colon\X\times\R\to\X$.  For $t\in(-\delta,\delta)$ we will write $\Phi_t$ for $\Phi|_{\X\times\{t\}}$.

Consider the morphism
\[\Phi_+\Phi_-\colon\X\times(-\delta,\delta)\to\X\]
given by the composition
\[\X\times(-\delta,\delta)\xrightarrow{\Id_\X\times d}X\times(-\delta,\delta)\times(-\delta,\delta)\xrightarrow{\Phi\times\Id_{(-\delta,\delta)}}\X\times(-\delta,\delta)\xrightarrow{\Phi}\X\]
with $d(t)=(-t,t)$.  Thus $\Phi_+\Phi_-|_{\X\times\{t\}}=\Phi_t\Phi_{-t}$.  It is easy to check that $\Phi_+\Phi_-$ integrates the zero vector field on $\X$, and that $\Phi_+\Phi_-|_{\X\times\{0\}}={\Phi|_{\X\times\{0\}}}^2$, which is equipped with $e_\Phi^2\colon{\Phi|_{\X\times\{0\}}}^2\Rightarrow\Id_\X$.  Therefore by Theorem~\ref{IntegralTheorem} there is $\Lambda\colon\Phi_+\Phi_-\Rightarrow\pi_1$, and in particular
\begin{align*}
\lambda_{\tdf}\colon&\Phi_{\tdf}\Phi_{-\tdf}\Rightarrow\Id_\X,\\
\lambda_{-\tdf}\colon&\Phi_{-\tdf}\Phi_\tdf\Rightarrow\Id_\X.
\end{align*}

Now set $K=3\delta/2$, so that the open intervals $I_n=(nK-\delta,nK+\delta)$ cover $\R$, and the only nonempty intersections among them are $I_{n-1}\cap I_n=((3n-2)\delta/2,(3n-1)\delta/2)$.  Define
\[\Phi_n\colon\X\times I_n\to\X\]
to be the composition
\[\X\times I_n\xrightarrow{\Phi_{3\delta/4}^{2n}\times\tau_n}\X\times I\xrightarrow{\Phi}\X\]
for $n\geqslant 0$, and similarly for $n\leqslant 0$ but with $\Phi_{3\delta/4}^{2n}$ replaced by $\Phi_{-3\delta/4}^{-2n}$; here $\tau_n(t)=t-{3\delta n}/{2}$.

Note that $n\frac{3\delta}{2}-\tdf=(n-1)\frac{3\delta}{2}-\tdf$, and that there is a $2$-morphism
\begin{gather*}
\Lambda_n\colon\Phi_n|_{\X\times\{n\frac{3\delta}{2}-\tdf\}}=\Phi_{-\tdf}\Phi_{\tdf}^{2n}\Longrightarrow\Phi_{\tdf}^{2n-1}=\Phi_{n-1}|_{\X\times\{n\frac{3\delta}{2}-\tdf\}},\quad n>0,\\
\Lambda_n\colon\Phi_n|_{\X\times\{n\frac{3\delta}{2}-\tdf\}}=\Phi_{\tdf}\Phi_{-\tdf}^{-2n}\Longrightarrow\Phi_{-\tdf}^{-2n+1}=\Phi_{n-1}|_{\X\times\{n\frac{3\delta}{2}-\tdf\}},\quad n\leqslant 0,
\end{gather*}
constructed using $\lambda_{-\tdf}$, $\lambda_\tdf$ respectively.
By Theorem~\ref{IntegralTheorem} these $2$-morphisms extend to
\[\Lambda_n\colon\Phi_n|_{I_{n-1}\cap I_n}\Rightarrow\Phi_{n-1}|_{\X\times I_{n-1}\cap I_n}.\]
Since there are no nonempy triple intersections $I_i\cap I_j\cap I_k$ for $i,j,k$ distinct, the $\Phi_n$ and $\Lambda_n$ immediately yield $\Phi\colon\X\times\R\to\X$ with $M_n\colon\Phi|_{\X\times I_n}\Rightarrow\Phi_n$ such that $M_{n-1}|_{\X\times I_{n-1}\cap I_n}\circ M_n|_{\X\times I_{n-1}\cap I_n}=\Lambda_n$.  It follows that $\Phi$ integrates $X$ and that there is $e_\Phi\colon\Phi|_{\X\times\{0\}}\Rightarrow\Id_\X$ as required.
\end{proof}

\begin{proof}[Proof of Proposition~\ref{RestrictedFlowProposition}.]
Let $A\to\X$ be an \'etale atlas for $\X$.  Let $A'\to\X\times\R$ the atlas induced fom $A$ by $\Phi$.  Let $B\to\Y$ and $B'\to\Y\times\R$ be the atlases induced by $\Y\to\X$ and $\Y\times\R\to\X\times\R$ respectively.  Thus $A\hookrightarrow B$, $A'\hookrightarrow B'$ are embedded submanifolds.

We claim that the induced map $\tilde\Phi_0\colon A'\to A$ sends $B'$ into $B$.  Assuming this for the time being, it follows that $\tilde\Phi_1\colon A'\times_{\X\times\R}A'\to A\times_\X A$ sends $B'\times_{\Y\times\R}B'$ into $B\times_\X B$, and there is a commutative diagram of groupoid-morphisms
\[\xymatrix{
[B'\times_{\Y\times\R}B'\rightrightarrows B']\ar[r]\ar[d]  & [B\times_\X B\rightrightarrows B]\ar[d]\\
[A'\times_{\X\times\R}A'\rightrightarrows A']\ar[r]  & [A\times_\X A\rightrightarrows A]
}\]
in which all but the top map are induced by the maps in 
\[\xymatrix{
\Y\times\R\ar[d]& \Y\ar[d]\\
\X\times\R\ar[r]^-\Phi& \X.
}\]
It follows that this last diagram can be completed to a $2$-commutative diagram
\[\xymatrix{
\Y\times\R\ar[d]\ar[r]^-\Psi& \Y\ar[d]_{}="1"\\
\X\times\R\ar[r]_-\Phi^{}="2"& \X\ar@{=>}"1";"2"
}\]
in which, by construction, $\Psi$ integrates $X|_\Y$.  The rest of the proposition now follows from Lemma~\ref{IdentityEmbeddingLemma}.

We now prove our claim that $\tilde\Phi_0\colon A'\to A$ sends $B'$ into $B$.  First note that, since $\Phi$ integrates $X$, $\tilde\Phi_0$ sends integral curves for $\ddt$ into integral curves for $X_A$, and that $X_A$ is tangent to the submanifold $B\hookrightarrow A$.  Moreover, since there is $e_\Phi\colon\Phi|_{\X\times\{0\}}\Rightarrow\Id_\X$, $\tilde\Phi_0$ sends the points in $B'$ that lie over $0\in\R$ into $B$.

Let $b'\in B'$ be some point, lying over time $t$, and without loss assume $t>0$.  By considering the morphism $\R\to \X\times\R$, $s\mapsto(\mathrm{Im}(b'),s)$, we can find $0=t_0<\cdots<t_n=t$, $\ddt$-integral curves $\gamma_i\colon[t_{i-1},t_i]\to B'$, and $2$-morphisms $\gamma_i(t_i)\Rightarrow\gamma_{i+1}(t_i)$ for all $1\leqslant i\leqslant n-1$, and $\gamma_n(t_n)\Rightarrow b'$.  Since $\gamma_i(0)$ lies over $0\in\R$, $\tilde\Phi_0(\gamma_i(0))$ lies in $B$.  Since each $\gamma_i$ is a $\ddt$ integral curve, if $\tilde\Phi_0(\gamma_i(t_{i-1})$ lies in $B$, so does $\tilde\Phi_0(\gamma_i(t_i)$.  Since $B\subset A$ is closed under $2$-morphisms, if $\tilde\Phi_0(\gamma_i(t_{i}))$ lies in $B$, so does $\tilde\Phi_0(\gamma_{i+1}(t_i))$.  The claim follows.
\end{proof}

\begin{lemma}\label{IdentityEmbeddingLemma}
Let $i\colon\Y\to\X$ be an embedding of differentiable stacks and suppose given a $2$-commutative diagram 
\[\xymatrix{
\Y\ar[d]_i\ar[r]^\Psi &\Y\ar[d]_{}="1"^i\\
\X\ar[r]_{\Id_\X}^{}="2" & \X.\ar@{=>}"1";"2"_{\lambda}
}\]
Then there is $\mu\colon\Psi\Rightarrow\Id_\Y$ such that $i_\ast\mu=\lambda$.
\end{lemma}
\begin{proof}
Let $A\to\X$ be an atlas.  Pulling back the above diagram under this atlas yields a commutative square of manifolds, which must be
\[\xymatrix{B\ar[r]^{\Id}\ar[d]& B\ar[d]\\
A\ar[r]_{\Id} & A,}\]
where $B\to A$ is the submanifold induced by $i$.  Since $B\to A$ and $B\times_\Y B\to A\times_\X A$ are embeddings and the square
\[\xymatrix{A\ar[r]^=\ar[d] & A\ar[d]_{}="1"\\
\X\ar[r]_{\Id_\X}^{}="2"& \X\ar@{=>}"1";"2"}\]
induces the identity map on $[A\times_\X A\rightrightarrows A]$, the square
\[\xymatrix{B\ar[r]^=\ar[d] & B\ar[d]_{}="1"\\
\Y\ar[r]_\Psi^{}="2"& \Y\ar@{=>}"1";"2"_m}\]
must induce the identity map on $[B\times_\Y B\rightrightarrows B]$, so that there is indeed a $2$-morphism $\mu\colon\Psi\Rightarrow\Id_\X$ that when composed with $m$ becomes trivial.  Since composing $m$ and $\lambda$ yields the trivial $2$-morphism, it follows that $i_\ast\mu$ and $\lambda$ coincide when pulled back to $B$, and the result follows.
\end{proof}

\begin{proof}[Proof of Lemma~\ref{UnderlyingFlowLemma}.]
Let $\pi_U\colon U\to\X$ be \'etale with $x_U\in U$ such that $x=[x_U\to\X]$.  Then $t\mapsto\bar f\circ\varphi_t(x)$ is the composition
\[\bar f\circ\bar\Phi\circ\overline{x_U\times\Id}=f\circ\Phi\circ(x_U\times\Id),\]
where we have written $x_U\times\Id\colon\R\to\X\times\R$.  This is certainly smooth.  Now let us compute the derivative of the composition; without loss assume $t=0$.  By the definition of what it means for a morphism to integrate $X$, we may find $\epsilon>0$ and a commutative diagram
\[\xymatrix{
(-\epsilon,\epsilon)\ar[d]_{x_U\times\mathrm{incl}}\ar[r]^\gamma & U\ar[d]_{}="1"^{\pi_U}\\ \X\times\R\ar[r]_\Phi^{}="2" & \X\ar@{=>}"1";"2"}\]
where $\gamma$ is the integral curve of $X_U$ through $x_U$.  Then
\[\left.\frac{d}{dt}\right |_{t=0}\bar f\circ\varphi_t(x)=\left.\frac{d}{dt}\right|_{t=0}f\circ\pi_U\circ\gamma=X_U\cdot (f\circ\pi_U)(x_U)=\overline{X\cdot f}(x)\]
as required.
\end{proof}

\section{The Strong Topology on $C^\infty(\X)$.}\label{StrongTopologySection}

Let $\X$ be a differentiable Deligne-Mumford stack.  The morphisms $f\colon\X\to\R$ form a set that we denote $C^\infty(\X)$.   This coincides with the set of those smooth functions $X\to\R$ on an atlas $X\to\X$ for which the two compositions $X\times_\X X\rightrightarrows X\to\R$ coincide.  This section will define the \emph{strong topology} on $C^\infty(\X)$ and study its properties.  We will show that $C^\infty(\X)$ is a Baire space in which the Morse functions form a dense open subset.  Thus Morse functions are abundant in a very precise sense.

It is usual to define the strong topology on $C^r(M,N)$ for manifolds $M,N$ and $0\leqslant r\leqslant\infty$ \cite{\Hirsch}. When $\X$ is a manifold $M$ our strong topology on $C^\infty(\X)$ coincides with the usual strong topology on $C^\infty(M,\R)$.

\begin{question}
Is it possible to define a topological stack $C^r(\X,\Y)$, `the mapping stack with the strong topology', for any differentiable Deligne-Mumford stacks $\X$, $\Y$ and $0\leqslant r\leqslant\infty$?
\end{question}

In \S\ref{StrongTopologyDefinitionSection} we define the strong topology on $C^\infty(\X)$ and verify that when $\X$ is a manifold our definition coincides with the usual one.  In \S\ref{StrongTopologyAtlasSection} we compare the strong topology on $C^\infty(\X)$ with the strong topology on $C^\infty(X)$ for an \'etale atlas $X\to\X$ and give a simple description of the strong topology on $C^\infty([M/G])$ for global quotients $[M/G]$.  In \S\ref{BaireSection} we show that $C^\infty(\X)$ is a Baire space, generalizing a familiar result for manifolds.  Finally, \S\ref{MorseDensitySection} is given to proving that Morse functions form a dense open subset of $C^\infty(\X)$ and concluding that any differentiable Deligne-Mumford stack $\X$ admits a Morse function $f$ with compact sublevel sets.  The existence of such Morse functions is crucial to the applications in Section~\ref{ApplicationsSection}.

\subsection{Definition of the strong topology.}\label{StrongTopologyDefinitionSection}

Recall that a family of morphisms $\alpha_i\colon A_i\to\X$ is called \emph{locally finite} if for each $U\to\X$ and each $u\in U$ there is a neighbourhood $V$ of $u$ such that $V\times_\X A_i$ is non-empty for only finitely many $i$.  

\begin{definition}\label{OpenSetsDefinition}
Suppose given the following data:
\begin{enumerate}
\item A locally finite family of \'etale morphisms $c_i\colon U_i\to\X$ from open subsets of $\R^n$, indexed by some set $I$.
\item A compact subset $K_i\subset U_i$ for each $i\in I$.
\item A positive number $\epsilon_i>0$ for each $i\in I$.
\item A non-negative integer $r$.
\item A morphism $f\colon\X\to\R$.
\end{enumerate}
Then we will write
\[\mathcal{N}^r(c_i,K_i,\epsilon_i,f)\subset C^\infty(\X)\]
for the set of all morphisms $g\colon\X\to\R$ satisfying the condition
\[\left|D^\mathbf{r}g\circ c_i(k)-D^\mathbf{r}f\circ c_i(k)\right|<\epsilon_i\mathrm{\ for\ all\ }i\in I,\ k\in K_i,\ |\mathbf{r}|\leqslant r.\]
Here $\mathbf{r}=(r_1,\ldots,r_m)$ is a list of numbers in $\{1,\ldots,n\}$, its length $|\mathbf{r}|$ is $m$, and $D^\mathbf{r}$ denotes the partial derivative $\frac{\partial^{|\mathbf{r}|}}{\partial x_{r_1}\cdots\partial x_{r_m}}$.
\end{definition}

\begin{definition}\label{StrongTopologyDefinition}
The \emph{strong topology} on $C^\infty(\X)$ is the topology with basis given by the sets $\mathcal{N}^r(c_i,K_i,\epsilon_i,f)$.
\end{definition}

When $\X$ is a manifold $M$, the definition of the strong topology on $C^\infty(M)$ is well known \cite[Chapter 2]{\Hirsch}.  We will now verify that the sets of Definition~\ref{OpenSetsDefinition} do indeed form the basis for a topology, and that in the case $\X=M$ we recover the usual definition of the strong topology.

\begin{lemma}\label{StrongTopologyBasisLemma}
The subsets $\mathcal{N}^r(c_i,K_i,\epsilon_i,f)$ form the basis for a topology on $C^\infty(\X)$.
\end{lemma}
\begin{proof}
Suppose we are given $\mathcal{N}^r(c_i,K_i,\epsilon_i,f)$ and $\mathcal{N}^s(d_j,J_j,\delta_j,g)$ and $h\in\mathcal{N}^r(c_i,K_i,\epsilon_i,f)\cap\mathcal{N}^s(d_j,J_j,\delta_j,g)$, where the $c_i\colon U_i\to\X$ are indexed by $i\in I$ and the $d_j\colon V_j\to\X$ are indexed by $j\in J$.  We will find a subset $\mathcal{N}^{t}(e_k,L_k,\phi_k,h)$ contained in $\mathcal{N}^r(c_i,K_i,\epsilon_i,f)\cap\mathcal{N}^s(d_j,J_j,\delta_j,g)$; this will show that $\mathcal{N}^r(c_i,K_i,\epsilon_i,f)\cap\mathcal{N}^s(d_j,J_j,\delta_j,g)$ is a union of basic open sets, as required.

For each $i\in I$ set $\epsilon_i'=\epsilon_i-\sup\{|D^\mathbf{r}f\circ c_i(x)-D^\mathbf{r}h\circ c_i(x)|\mid |\mathbf{r}|\leqslant r,\ x\in K_i\}$.  Note that $\epsilon'_i>0$ since $K_i$ is compact.  If $l\colon\X\to\R$ satisfies $|D^\mathbf{r}l\circ c_i(x)-D^\mathbf{r}h\circ c_i(x)|<\epsilon_i'$ whenever $x\in K_i$ and $|\mathbf{r}|\leqslant r$, then it follows that $|D^\mathbf{r}l\circ c_i(x)-D^\mathbf{r}f\circ c_i(x)|<\epsilon_i$ whenever $x\in K_i$ and $|\mathbf{r}|\leqslant r$.

Similarly, for each $j\in J$ set $\delta_j'=\delta_j-\sup\{|D^\mathbf{r}g\circ d_j(x)-D^\mathbf{r}h\circ d_j(x)|\mid |\mathbf{r}|\leqslant r,\ x\in J_j\}$.  Note that $\delta'_j>0$ since $J_j$ is compact.  If $l\colon\X\to\R$ satisfies $|D^\mathbf{r}l\circ d_j(x)-D^\mathbf{r}h\circ d_j(x)|<\delta_j'$ whenever $j\in J_i$ and $|\mathbf{r}|\leqslant s$, then it follows that $|D^\mathbf{r}l\circ d_j(x)-D^\mathbf{r}f\circ d_j(x)|<\delta_j$ whenever $x\in J_j$ and $|\mathbf{r}|\leqslant s$.

Now consider the open set $\mathcal{N}^t(e_k,L_k,\phi_k,h)$, where $t=\max\{r,s\}$, where the indexing set is $I\cup J$, and where
\begin{eqnarray*}
(e_k\colon W_k\to\X)&=&\left\{\begin{array}{cc} c_k\colon U_k\to\X & \mathrm{if\ }k\in I,\\ d_k\colon V_k\to\X & \mathrm{if\ }k\in J,\end{array}\right.\\
L_k&=&\left\{\begin{array}{cc} K_k & \mathrm{if\ }k\in I,\\ J_k & \mathrm{if\ }k\in J,\end{array}\right.\\
\phi_k&=&\left\{\begin{array}{cc} \epsilon'_k & \mathrm{if\ }k\in I,\\ \delta'_k & \mathrm{if\ }k\in J.\end{array}\right.
\end{eqnarray*}
It is clear from the previous two paragraphs that
\[\mathcal{N}^t(e_k,L_k,\phi_k,h)\subset \mathcal{N}^r(c_i,K_i,\epsilon_i,f)\cap\mathcal{N}^s(d_j,J_j,\delta_j,g),\]
as claimed.
\end{proof}

\begin{proposition}\label{StrongTopologyForManifoldProposition}
Let $\X=M$ for some manifold $M$.  Then the topology on $C^\infty(M)$ given in Definition~\ref{StrongTopologyDefinition} coincides with the usual notion of the strong topology on $C^\infty(M,\R)$.
\end{proposition}
\begin{proof}
Let us recall from \cite[Chapter 2]{\Hirsch} that the the strong topology on $C^\infty(M)$ is the topology with basis given by open sets
\begin{equation}\label{UsualOpenSetEquation}\mathcal{N}^r(c_i,K_i,d_i,\epsilon_i,f)\end{equation}
where $c_i\colon U_i\to M$ is a locally finite family of charts on $M$ (i.e.~open embeddings from open subsets of $\R^n$), $K_i\subset U_i$ are compact subsets, $d_i\colon V_i\to\R$ are charts, $\epsilon_i$ are positive real numbers, $r$ is a non-negative integer, and $f\colon M\to\R$ is a map for which $f(c_i(K_i))\subset d_i(V_i)$ for all $i$; then $\mathcal{N}^r(c_i,K_i,d_i,\epsilon_i,f)$ denotes the set of all $g\colon M\to\R$ for which $g(c_i(K_i))\subset d_i(V_i)$ and 
\[|D^\mathbf{s}(d_i^{-1}\circ g\circ c_i)(k)-D^\mathbf{s}(d_i^{-1}\circ f\circ c_i)(k)|<\epsilon_i\]
for all $\mathbf{s}$ with $|\mathbf{s}|\leqslant r$ and all $k\in K_i$.

In order to prove the lemma we will show first that we can assume, for the sets $\mathcal{N}^r(c_i,K_i,d_i,\epsilon_i,f)$ above, that the $d_i\colon V_i\to\R$ are all simply $\Id\colon\R\to\R$.  Thus the strong topology in its usual definition is generated by those sets $\mathcal{N}^r(c_i,K_i,\epsilon_i,f)$ from Definition~\ref{OpenSetsDefinition} for which the $c_i$ are all embeddings rather than just \'etale maps.  We will then show that any open set $\mathcal{N}^r(c_i,K_i,\epsilon_i,f)$ from Definition~\ref{OpenSetsDefinition} can be written as $\mathcal{N}^r(c'_j,K'_j,\epsilon'_j,f)$ where the $c_j'\colon U_j'\to M$ are open embeddings, thus completing the proof.

So suppose given $\mathcal{N}^r(c_i,K_i,d_i,\epsilon_i,f)$ as above.  We may take compact neighbourhoods $J_i$ of each $f(c_i(K_i))$ with $J_i\subset d_i(V_i)$ and -- by decreasing the $\epsilon_i$ if necessary -- assume that $g(K_i)\subset J_i$ for all $i$ and any $g\in\mathcal{N}^r(c_i,K_i,d_i,\epsilon_i,f)$.  Now set
\[\lambda_i=\sup_{j\in J_i,\,n\leqslant r,\,m\leqslant r}\left|\frac{d^nd_i^{-1}}{dx^n}(j)\right|^m\]
and define $\delta_i=\epsilon_i/(2^r\lambda_i)$.  Then for $g\in\mathcal{N}^r(c_i,K_i,\Id,\delta_i,f)$ and $k\in K_i$ we have
\[|D^\mathbf{s} d_i^{-1}\circ(g\circ c_i)(k)-D^\mathbf{s}d_i^{-1}\circ (f\circ c_i)(k)|<2^r\cdot\lambda_i\cdot\delta_i=\epsilon_i\]
so that $g\in\mathcal{N}^r(c_i,K_i,d_i,\epsilon_i,f)$.  Thus any set of the form \eqref{UsualOpenSetEquation} is a union of those of the form $\eqref{UsualOpenSetEquation}$ for which each $d_i=\Id_\R$.

Now suppose given a basic open set $\mathcal{N}^r(c_i,K_i,\epsilon_i,f)\subset C^\infty(M)$ as in Definition~\ref{OpenSetsDefinition}.  For each $i$ and each $k\in K_i$ we may find $E_k\supset B_k\owns k$, where $B_k$ is a closed ball and $E_k$ is an open ball small enough that $c_i|\colon E_k\to M$ is an open inclusion.  Take finitely many $k_i^m\in K_i$ for which $K_i\subset\bigcup B_{k_i^m}$.  Now set $c_i^m\colon U_i^m\to M$ to be the restriction of $c_i$ to $U_i^m=E_{k_i^m}$, set $K_i^m=K_i\cap B_{k_i^m}$, and finally set $\epsilon_i^m=\epsilon_i$.  Then the $c_i^m$, taken for all $i\in I$ and all $m$, are a locally finite family and 
\[\mathcal{N}^r(c_i,K_i,d_i,\epsilon_i,f)=\mathcal{N}^r(c_i^m,K_i^m,\epsilon_i^m,f).\]
This completes the proof.
\end{proof}

\subsection{Relation to the strong topology on an atlas.}\label{StrongTopologyAtlasSection}

Suppose we have an \'etale atlas $\pi\colon X\to\X$.  Write $C^\infty(X)^\mathrm{inv}$ for the set of smooth functions $X\to\R$ for which the two compositions $X\times_\X X\rightrightarrows X\to\R$ coincide.  Composition with $\pi$ determines a bijection
\[\pi^\ast\colon C^\infty(\X)\to C^\infty(X)^\mathrm{inv}.\]
Since $C^\infty(\X)$ admits the strong topology and $C^\infty(X)^\mathrm{inv}$ inherits a topology from the strong topology on $C^\infty(X)$, it is natural to ask whether the two topologies coincide under $\pi^\ast$.  The answer is that they do not coincide, and that in general $(\pi^\ast)^{-1}$ is continuous but $\pi^\ast$ is not.  For example, the bijection $C^\infty(S^1)\to C^\infty(\R)^\mathrm{inv}$ associated to the exponential map $\R\to S^1$ is not continuous.  However, in the very restrictive case that $\pi$ is proper, $\pi^\ast$ is indeed a homeomorphism.  This applies to the tautological atlases $M\to[M/G]$ for global quotients $[M/G]$ with $G$ finite.

\begin{proposition}\label{StrongTopologyAtlasProposition}
Let $\pi\colon X\to\X$ be an \'etale atlas and consider the bijection
\[\pi^\ast\colon C^\infty(\X)\to C^\infty(X)^\mathrm{inv}.\]
Then $(\pi^\ast)^{-1}$ is continuous.  If $\pi$ is proper then $\pi^\ast$ is also continuous.  In particular, the strong topology on $C^\infty([M/G])=C^\infty(M)^G$ for $G$ finite is simply the topology inherited from the strong topology on $C^\infty(M)$.
\end{proposition}
\begin{proof}
Let $\mathcal{N}^{r}(c_i,K_i,\epsilon_i,f)$ be a basic open neighbourhood in $C^\infty(\X)$.  Consider the diagrams
\[\xymatrix{
U_i\times_\X X \ar[r]^-{\tilde c_i}\ar[d]_{\pi_i} & X\ar[d]^\pi_{}="1"\\
U_i\ar[r]_{c_i}^{}="2"  &\X.\ar@{=>}"1";"2"
}\]
For each $k\in K_i$ we may find open discs $V_k\subset W_k$ centered at $k$, for which $\cl V_k\subset W_k$, and for which $\pi_i$ admits a local section $\pi_i^{-1}\colon W_k\to U_i\times_\X X$.  Choose finitely many $k_i^m$ for which the $V_i^m=V_{k_i^m}$ cover $K_i$.  Set $W_i^m=W_{k_i^m}$ and $\epsilon_i^m=\epsilon_i$.  Now it is simple to verify that the family of \'etale maps 
\[c_i^m\colon W_i^m\xrightarrow{\pi_i^{-1}}\pi_i^{-1}(W_i^m)\xrightarrow{\tilde c_i}X\]
is locally finite, that $K_i\cap\cl V_i^m$ is compact, and that $g\in\mathcal{N}^{r}(c_i,K_i,\epsilon_i,f)$ if and only if $g\circ\pi\in\mathcal{N}^{r}(c_i^m,K_i\cap\cl V_i^m,\epsilon_i^m,f\circ\pi)$.  This shows that $(\pi^\ast)^{-1}$ is continuous. 

Now let $\pi$ be proper and suppose given a basic open subset $\mathcal{N}^{r}(d_j,J_j,\delta_j,g\circ\pi)$ in $C^\infty(X)$.  We claim that the $\pi\circ d_j\colon V_j\to\X$ again form a locally finite family.  Assuming this for the time being, it is immediate that $h\circ\pi\in\mathcal{N}^{r}(d_j,J_j,\delta_j,g\circ\pi)$ if and only if $h\in\mathcal{N}^{r}(\pi\circ d_j,J_j,\delta_j,g)$, so that $\pi^\ast$ is continuous.

Now we prove our claim.  Let $U\to\X$ be any morphism and let $u\in U$.  Then since $U\times_\X X\to U$ is proper \'etale and the $d_j\colon V_j\to X$ are locally finite, we may find an open neighbourhood $W$ of $u$ such that the open set $\pi_2(\pi_1^{-1}(W))\subset X$ meets only finitely many of the $d_j(V_j)$.  It follows that only finitely many of the $W\times_\X V_j$ are non-empty, as required.
\end{proof}

\subsection{$C^\infty(\X)$ is a Baire space.}\label{BaireSection}

\begin{theorem}\label{BaireTheorem}
$C^\infty(\X)$ with the strong topology is a Baire space.  That is, a countable union of dense open subsets in $C^\infty(\X)$ is again dense.
\end{theorem}

Before proving this theorem we will establish the next proposition, which gives us a relatively `lean' description of the strong topology.

\begin{proposition}\label{CoverTopologyProposition}
Let $s_l\colon S_l\to\X$ be a countable locally finite family of \'etale morphisms  from open subsets of $\R^n$, indexed by $l\in L$, together with open subsets $T_l\subset S_l$ for which $\cl T_l$ is compact and $\bigsqcup s_l\colon\bigsqcup T_l\to\X$ is surjective.  Then the open subsets $\mathcal{N}^r(s_l,\cl T^l,\epsilon_l,f)$ form a basis for the strong topology on $C^\infty(\X)$.
\end{proposition}

By Proposition~\ref{SpecialCoverProposition} we can always find a family of \'etale morphisms that satisfies the hypotheses of Proposition~\ref{CoverTopologyProposition}.

\begin{proof}
It is simple to show, much as in the proof of Lemma~\ref{StrongTopologyBasisLemma}, that the $\mathcal{N}^r(s_l,\cl T^l,\epsilon_l,f)$ do form the basis for a topology on $C^\infty(\X)$ which is no finer than the strong topology.  Now let $\mathcal{N}^r(c_i,K_i,\epsilon_i,f)$ be a basic open subset of $C^\infty(\X)$, where the $c_i\colon U_i\to\X$ are indexed by $I$.  We will find a $\delta_l>0$ for each $l\in L$, such that
\[\mathcal{N}^r(s_l,\cl T_l, \delta_l,f)\subset \mathcal{N}^r(c_i,K_i,\epsilon_i,f).\]
This will show that the $\mathcal{N}^r(s_l,\cl T_l,\delta_l,f)$ are a basis for the strong topology.

For each $l\in L$ and $i\in I$ we have a diagram
\[\xymatrix{
U_i\times_\X S_l\ar[r]^-{\pi_2}\ar[d]_{\pi_1} & S_l\ar[d]^{s_l}_{}="1"\\
U_i\ar[r]_{c_i}^{}="2" &\X.\ar@{=>}"1";"2"
}\]
By Lemma~\ref{EstimateLemma} below we may find $M_{il}>0$ such that 
\begin{equation}
|D^\mathbf{s}(g\circ c_i)(\pi_1(k))|\leqslant M_{il}\max_{|\mathbf{t}|\leqslant|\mathbf{s}|} |D^\mathbf{t}(g\circ s_l)(\pi_2(k))|\label{SecondEstimateEquation}
\end{equation}
for all $g\colon\X\to\R$, $k\in(\pi_1\times\pi_2)^{-1}(K_i\times\cl T_l)$, and $\mathbf{s}$ with $|\mathbf{s}|\leqslant r$.

Now for each $l\in L$ set
\[\delta_l=\min\left(\frac{\epsilon_i}{M_{il}}\right)>0,\]
where the minimum is taken over the finitely many $i$ for which $(\pi_1\times\pi_2)^{-1}(K_i\times\cl  T_l)$ is non-empty.  Let $g\in\mathcal{N}^r(s_l,\cl T_l,\delta_l,f)$.   Since $\bigsqcup s_l$ is surjective, for any $i\in I$ and any $u\in K_i$ there is $l\in L$ and $k\in (\pi_1\times\pi_2)^{-1}(K_i\times\cl T_l)$ such that $u=\pi_1(k)$.  It now follows from \eqref{SecondEstimateEquation} that $g\in\mathcal{N}^r(c_i,K_i,\epsilon_i,f)$ as required.
\end{proof}

\begin{lemma}\label{EstimateLemma}
Let $\rho\colon U\to\X$, $\sigma\colon V\to\X$ be \'etale maps from open subsets of $\R^n$ and let $J\subset U$, $K\subset V$ be compact so that we have a diagram
\[\xymatrix{
U\times_\X V\ar[r]^-{\pi_2}\ar[d]_{\pi_1} & V\ar[d]^{\sigma}_{}="1"\\
U\ar[r]_{\rho}^{}="2" &\X.\ar@{=>}"1";"2"
}\]
Then given $r\geqslant 0$ there is a constant $M>0$ such that
\[|D^\mathbf{s}(g\circ\rho)(\pi_1(k))|\leqslant M\max_{|\mathbf{t}|\leqslant r} |D^\mathbf{t}(g\circ\sigma)(\pi_2(k))|\]
for all $k\in(\pi_1\times\pi_2)^{-1}(J\times K)$, $g\colon\X\to\R$, and $\mathbf{s}$ with $|\mathbf{s}|\leqslant r$.
\end{lemma}
\begin{proof}
We will use the following estimate.  Suppose given open subsets $A,B$ of $\R^n$ and smooth functions $\alpha\colon A\to\R$, $\beta\colon B\to A$.  Then for $b\in B$
\begin{equation}\label{EstimateEquation}|D^\mathbf{s}(\alpha\circ\beta)(b)|\leqslant 2^{|\mathbf{s}|}\max_{|\mathbf{t}|\leqslant |\mathbf{s}|}|D^\mathbf{t} \alpha(\beta(b))|\max_{|\mathbf{t}|\leqslant |\mathbf{s}|,m\leqslant |\mathbf{s}|}|D^\mathbf{t}\beta(b)|^m.\end{equation}

Now since $\pi_1$ and $\pi_2$ are \'etale, each $k\in(\pi_1\times\pi_2)^{-1}(J\times K)$ has an open neighbourhood $W_k$ such that $\pi_1|_{W_k}$ and $\pi_2|_{W_k}$ are diffeomorphisms onto their images.  Set $b_k=(\pi_2|_{W_k})(\pi_1|_{W_k})^{-1}$. Now set
\[M=2^{r}\sup\{|D^\mathbf{t} b_k(k)|^m\,:\, {k\in(\pi_1\times\pi_2)^{-1}(J\times K) ,\,|\mathbf{t}|\leqslant r,\,m\leqslant r} \}.\]
Since $D^\mathbf{t} b_k(k)$ depends continuously on $k$ and $(\pi_1\times\pi_2)^{-1}(J\times K)$ is compact, this supremum does indeed exist.  Further, for $g\colon\X\to\R$ and $k$ as above, \eqref{EstimateEquation} gives us
\begin{eqnarray*}
|D^\mathbf{s}(g\circ\rho)(\pi_1(k))|
&=&|D^\mathbf{s}(g\circ\sigma\circ b_k)(\pi_1(k))|\\
&\leqslant&M\max_{|\mathbf{t}|\leqslant r} |D^\mathbf{t}(g\circ\sigma)(\pi_2(k))|
\end{eqnarray*}
as required.
\end{proof}

\begin{proof}[Proof of Theorem~\ref{BaireTheorem}.]
Choose a locally finite family of \'etale morphisms $s_l\colon S_l\to\X$ from open subsets of $\R^n$, together with open subsets $T_l\subset S_l$ satisfying the hypothesis of Proposition~\ref{CoverTopologyProposition}, so that the sets
\[\mathcal{N}^r(\epsilon_l,f)=\mathcal{N}^r(s_l,\cl T_l,\epsilon_l,f)\]
form a basis for the strong topology.

Let $A_i$, $i=0,1,\ldots$ be a sequence of dense open subsets of $C^\infty(\X)$, and let $U\subset C^\infty(\X)$ be open.  We must show that $\bigcap_i A_i\cap U$ is nonempty.

Since $A_0$ is dense and open, $A_0\cap U$ is non-empty and open, so we can find $f_0$, $\epsilon^{(0)}_l$ and $r_0$ such that $\mathcal{N}^{r_0}(2\epsilon^{(0)}_l,f_0)\subset A_0\cap U$.  Since $A_1$ is dense and open, we may find $f_1$, $\epsilon^{(1)}_l$ and $r_1$ such that $\mathcal{N}^{r_1}(2\epsilon^{(1)}_l,f_1)\subset A_1\cap\mathcal{N}^{r_0}(\epsilon^{(0)}_l,f_0)$.  Without loss we may assume that $r_1>r_0$ and that $\epsilon^{(1)}_l\leqslant{\epsilon^{(0)}_l}/{2}$ for $l\in L$.  Proceeding inductively we find sequences $f_i$, $\epsilon^{(i)}_l$ and $r_i$ such that
\begin{gather*}
\mathcal{N}^{r_i}(2\epsilon^{(i)}_l,f_i)\subset A_i\cap\mathcal{N}^{r_{i-1}}(\epsilon^{(i-1)}_l,f_{i-1}),\\
\epsilon^{(i)}_l\leqslant\frac{\epsilon^{(i-1)}_l}{2},\\
r_i>r_{i-1},
\end{gather*}
for $i=1,2,\ldots$.

Now fix $l\in L$ and consider the sequence of smooth functions $(g_i)_{i=0}^\infty$ given by $g_i=f_i\circ s_l\colon S_l\to\R$.  By construction,
\[\left|D^{\mathbf{s}}g_i(t)-D^{\mathbf{s}}g_{i-1}(t)\right|<\frac{\epsilon^0_l}{2^{i-1}}\]
for all $\mathbf{s}\in\cl T_l$ and all $i\geqslant |\mathbf{s}|+1$.  Thus for each $\mathbf{s}$ the sequence $\frac{\partial^\mathbf{s}}{\partial x^\mathbf{s}}g_i|_{T_l}$ is Cauchy in the $\sup$-norm, and it follows that $(g_i|_{T_l})$ converges in the $\sup$-norm to a smooth function $g^{(l)}\colon T_l\to\R$.  

Now consider the diagram
\[\xymatrix{
T_l\times_\X T_m\ar[r]^-{\pi_1}\ar[d]_{\pi_2} &  T_l\ar[d]^{s_l|_{T_l}}_{}="1"\\
T_m\ar[r]_{s_m|_{T_m}}^{}="2" & \X
}\]
and note that for $k\in T_l\times_\X T_m$ we have
\begin{eqnarray*}
g^{(l)}\circ\pi_1(k)
&=&\lim_{i\to\infty}f_i\circ s_l\circ\pi_1(k)\\
&=&\lim_{i\to\infty}f_i\circ s_m\circ\pi_2(k)\\
&=&g^{(m)}\circ\pi_2(k),
\end{eqnarray*}
and therefore, since $\bigsqcup s_l|_{T_l}\colon\bigsqcup T_l\to\X$ is surjective, the $g^{(l)}$ patch to give a function $g\colon\X\to\R$ such that $g\circ s_l|_{T_l}=g^{(l)}$ for each $l\in L$.  

We claim that $g\in\bigcap A_i\cap U$.  First note that, since $f_{j+1}\in\mathcal{N}^{r_j}(\epsilon^j_l,f_j)$ for each $j\geqslant 0$, there is for each $l\in\mathbb{N}$ and $j\geqslant 0$ some $0< k^j_l<\epsilon^j_l$ such that
\[\left|D^\mathbf{s}f_j\circ s_l(t)-D^\mathbf{s}f_{j+1}\circ s_l(t)\right|\leqslant k^j_l\]
for all $|\mathbf{s}|\leqslant r_j$, $t\in T_l$.   Therefore
\begin{eqnarray*}
\left|D^\mathbf{s}f_j\circ s_l(t)-D^\mathbf{s}g\circ s_l(t)\right|
&=&\lim_{j'\to\infty}\left|D^\mathbf{s}f_j\circ s_l(t)-D^\mathbf{s}f_{j'}\circ s_l(t)\right|\\
&\leqslant& \sum_{m=0}^{\infty}k_l^{j+m},
\end{eqnarray*}
for all $t\in T_l$ and $|\mathbf{s}|\leqslant r_j$; the same then holds for all $t\in\cl T_l$ and $|\mathbf{s}|\leqslant r_j$, and since
\[\sum_{m=0}^{\infty}k_l^{j+m}<\sum_{m=0}^{\infty}\epsilon_l^{(j+m)}\leqslant \sum_{m=0}^{\infty}\epsilon_l^{(j)}/2^m=2\epsilon_l^{(j)}\]
it follows that $g\in\mathcal{N}^{r_j}(2\epsilon^j_l,f_l)\subset A_j\cap U$, as required.
\end{proof}

\subsection{Density of Morse functions.}\label{MorseDensitySection}

\begin{theorem}\label{MorseFunctionsAreDenseTheorem}
Morse functions form a dense open subset of $C^\infty(\X)$.
\end{theorem}

Thus every differentiable Deligne-Mumford stack admits a Morse function.  Certain applications will require Morse functions with a stronger property.  These are provided by the next proposition and its corollary.

\begin{proposition}\label{UnboundedFunctionsProposition}
Let $\mathcal{C}\subset C^\infty(\X)$ denote the subset consisting of morphisms $f\colon\X\to\R$ for which $\bar f^{-1}(-\infty,t]\subset\bar\X$ is compact for all $t$.  Then $\mathcal{C}$ is non-empty and open.
\end{proposition}

\begin{corollary}\label{ProperMorseFunctionCorollary}
Any differentiable Deligne-Mumford stack $\X$ admits a Morse function $f\colon\X\to\R$ with the property that $\bar f^{-1}(-\infty,t]$ is compact for each $t\in\R$.
\end{corollary}

\begin{proof}[Proof of Proposition~\ref{UnboundedFunctionsProposition}.]
This is immediate if $\bar\X$ is compact.  Suppose that $\bar\X$ is not compact.  We begin by showing that $\mathcal{C}$ is open.  Let $s_l\colon S_l\to\R$, $T_l\subset S_l$ be a countable cover of $\X$ as in Proposition~\ref{SpecialCoverProposition}.  Suppose that $f\in\mathcal{C}$ and let $g\in\mathcal{N}^0(s_l,\cl T_l,1,f)$.  Then if $x\in \bar g^{-1}(-\infty,t]$ it follows --- since $x=[s\to\X]$ for some $l\in L$ and some $s\in T_l$ --- that $x\in\bar f^{-1}(-\infty,t+1]$.  Thus $\bar g^{-1}(-\infty,t]$ is a closed subset of the compact $\bar f^{-1}(-\infty,t+1]$ and so is itself compact.  This shows that $\mathcal{C}$ is open.

Now we will show that $\mathcal{C}$ is non-empty.  As in the proof of \cite[1.9]{\Warner} we may find $G_1\subset G_2\subset\cdots\subset\bar\X$ open, such that each $\cl G_i$ is compact and contained in $G_{i+1}$ and such that $\bar\X=\bigcup G_i$.  Now using Theorem~\ref{PartitionsOfUnityTheorem} we may take a partition of unity $\phi_1,\phi_2,\ldots$ subordinate to $\{G_i\}$.  Set
\[\chi_j=1-\sum_{\supp\phi_i\cap\cl G_j\neq\emptyset}\phi_i.\]
Then 
\begin{enumerate}
\item $\supp\chi_j\subset\bar\X\setminus G_j$,
\item $\bar\chi_j=1$ on $\bar\X\setminus G_{k_j}$ for some $k_j$ large enough,
\item $\chi_i\geqslant \chi_{i+1}$ for all $i\geqslant 1$.
\end{enumerate}

Since each $x\in\bar\X$ has a neighbourhood which meets the support of only finitely many $\chi_j$ we may --- as in the proof of Theorem~\ref{PartitionsOfUnityTheorem} --- form the sum
\[\chi=\sum\chi_j.\]
Let $N\in\mathbb{N}$.  Then if $x\in\bar\chi^{-1}(-\infty,N-1]$ we must have $\bar\chi_N(x)<1$, so that $x\in G_{k_N}$.  Thus $\bar\chi^{-1}(-\infty,N-1]$ is a subset of $G_{k_N}$ and is therefore compact.  This is sufficient to show that $\bar\chi^{-1}(-\infty,t]$ is compact for all $t\in\R$.
\end{proof}

\begin{lemma}\label{ExtensionByZeroLemma}
Let $i\colon\Y\to\X$ be an open embedding of differentiable Deligne-Mumford stacks and suppose given $\phi\colon\Y\to\R$ with $\supp\phi$ compact.  The function $C^\infty(\Y)\to C^\infty(\X)$, $g\mapsto \widetilde{\phi g}$ is continuous.  (For $h\colon\Y\to\R$ with compact support, the extension by zero $\tilde h\colon\X\to\R$ was defined in Lemma~\ref{ExtensionByZeroExistsLemma}.)
\end{lemma}

\begin{proof}[Proof of Theorem~\ref{MorseFunctionsAreDenseTheorem}.]
Using Proposition~\ref{SpecialCoverProposition} we may find countably many linear orbifold-charts $i_l\colon[S_l/G_l]\to\X$ indexed by $l\in L$, with compact subsets $K_l\subset S_l$, such that the underlying maps $K_l\to S_l\to\bar\X$ cover $\bar\X$, and such that the $\overline{[S_l/G_l]}$ form a locally finite cover of $\bar\X$.  Write $s_l\colon S_l\to\X$.  Further, using Corollary~\ref{BumpFunctionsLemma}, take for each $l\in L$ a compactly supported $\phi_l\colon[S_l/G_l]\to\R$ such that $\bar\phi_l\circ s_l=1$ in a neighbourhood of $K_l$ and such that $\phi_l$ has compact support.

Let $f\colon\X\to\R$ be a Morse function, so that each $f\circ s_l$ is Morse.  Then it follows immediately from \cite[5.32]{\BanyagaHurtubise} that we may find $\delta_l>0$ for each $l$ such that any $g\in\mathcal{N}^2(s_l,K_l,\delta_l,f)$ has the property that each $g\circ s_l$ has no degenerate critical points in $K_l$.  But such a $g$ is then Morse.  Thus the Morse functions form an open subset of $C^\infty(\X)$.

Now write $\mathcal{M}_l\subset C^\infty(\X)$ for the subset consisting of functions $f$ for which $f\circ s_l$ has no degenerate critical points in $K_l$.  Thus $\bigcap_l\mathcal{M}_l$ is the subset consisting of all Morse functions.  Again by \cite[5.32]{\BanyagaHurtubise} each $\mathcal{M}_l$ is open, and we will prove that each $\mathcal{M}_l$ is dense.  It then follows from Theorem~\ref{BaireTheorem} that the set of Morse functions is itself dense.

Let $G$ be a finite group acting on a manifold $S$.  Write $C^\infty_G(S)^r$, $r\leqslant\infty$, for the set of $G$-invariant smooth functions on $S$ equipped with the topology inherited from the $C^r$ topology on $C^\infty(S)$.  Wasserman \cite[Lemma 4.8]{\Wasserman} has shown that the set of $G$-invariant Morse functions on $S$ form a dense subset of $C^\infty_G(S)^r$ for $r<\infty$.  Since the topology on $C^\infty_G(S)^\infty$ is given by the union of the topologies on $C^\infty_G(S)^r$ it follows that the $G$-invariant smooth Morse functions on $S$ form a dense subset of $C^\infty_G(S)^\infty$.  Using Proposition~\ref{StrongTopologyAtlasProposition} we can restate the above paragraph as follows: The Morse functions on $[S/G]$ form a dense open subset of $C^\infty([S/G])$.

We now return to our claim that the $\mathcal{M}_l$ are dense.  Let $f\in C^\infty(\X)$ and let $\mathcal{N}$ be an open neighbourhood of $f$.  We will find an element of $\mathcal{N}\cap\mathcal{M}_l$, and this will prove that $\mathcal{M}_l$ is dense.  By Lemma~\ref{ExtensionByZeroLemma} there is an open neighbourhood $\mathcal{N}'$ of $f\circ i_l\in C^\infty([S_l/G_l])$ with the property that, for $g\in\mathcal{N}'$, $f(1-\tilde\phi_l)+\widetilde{\phi g}$ lies in $\mathcal{N}$.  Now by our restatement of Wasserman's result there is a Morse function $g$ on $[S_l/G_l]$ that lies within $\mathcal{N}'$, so that $f(1-\tilde\phi_l)+\widetilde{\phi g}$ lies in $\mathcal{N}$.  Now by construction $[f(1-\tilde\phi_l)+\widetilde{\phi g}]\circ s_l=g$ in a neighbourhood of $K_l$, so that $f(1-\tilde\phi_l)+\widetilde{\phi g}\in\mathcal{M}_l\cap\mathcal{N}$.   This completes the proof.
\end{proof}

\begin{proof}[Proof of Lemma~\ref{ExtensionByZeroLemma}.]
Let $g\in C^\infty(\Y)$ and let $\mathcal{N}^r(c_i,K_i,\epsilon_i,\widetilde{g\phi})$ be a basic open neighbourhood of $\widetilde{g\phi}$, where the $c_i\colon U_i\to\X$ are indexed by $i\in I$.  We may, by a slight modification of the proof of Proposition~\ref{SpecialCoverProposition}, find finitely many \'etale morphisms $d_j\colon V_j\to\Y$ from open subsets of $\R^n$, together with compact subsets $J_j\subset V_j$, such that for each $j$ only finitely many of the $V_j\times_\X U_i$ are nonempty, and the images of the underlying maps $J_j\to\bar\Y$ cover $\supp\phi$.

Consider the diagram
\[\xymatrix{
U_i\times_\X V_j\ar[r]^-{\pi_2}\ar[d]_{\pi_1}  & V_j\ar[d]_{}="1"^{i\circ d_j}\\
U_i\ar[r]_{c_i}^{}="2" & \X.\ar@{=>}"1";"2"
}\]
By Lemma~\ref{EstimateLemma} we may find $M_{ij}>0$ such that for all $g\colon\X\to\R$, all $\mathbf{s}$ with $|\mathbf{s}|\leqslant r$, and all $k\in(\pi_1\times\pi_2)^{-1}(K_i\times J_j)$,
\[|D^\mathbf{s}(g\circ c_i)(\pi_1(k))|\leqslant M_{ij}\max_{|\mathbf{t}|\leqslant r}|D^\mathbf{t}(g\circ i\circ d_j)(\pi_2(k))|.\]
Further, set
\[F_j=\sup_{|\mathbf{t}|\leqslant r,\, x\in J_j,\, m\leqslant r}|D^\mathbf{t}(\phi\circ d_j)(x)|^m.\]

Now suppose given $h\colon\Y\to\R$, $\mathbf{s}$ with $|\mathbf{s}|\leqslant r$, and $k\in K_i$.  Then either $[k\to\X]$ lies outside $\supp\phi$, in which case $\widetilde{h\phi}\circ c_i=0$ in a neighbourhood of $k$, so that $|D^\mathbf{s}(\widetilde{h\phi}\circ c_i)(k)|=0$, or alternatively there is some $j$ and some $\tilde k\in (\pi_1\times\pi_2)^{-1}(K_i\times J_j)$ for which $k=\pi_1(\tilde k)$.  Then
\begin{eqnarray*}
|D^\mathbf{s}(\widetilde{h\phi}\circ c_i)(k)|
&\leqslant& M_{ij}\max_{|\mathbf{t}|\leqslant r}|D^\mathbf{t}(\widetilde{h\phi}\circ i\circ d_j)(\pi_2(\tilde k))|\\
&=& M_{ij}\max_{|\mathbf{t}|\leqslant r}|D^\mathbf{t}(h\circ d_j)(\phi\circ d_j)(\pi_2(\tilde k))|\\
&\leqslant& M_{ij}2^rF_j\max_{|\mathbf{t}|\leqslant r}|D^\mathbf{s}(h\circ d_j)(\pi_2(\tilde k))|.
\end{eqnarray*}

Set $\delta_j=\max\frac{\epsilon_i}{M_{ij}2^rF_j}$, where the maximum is taken over $i$ for which $U_i\times_\X V_j$ is non-empty.  The estimate above then shows that $\mathcal{N}^r(d_j,J_j,\delta_j,g)$ lies in the preimage of $\mathcal{N}^r(c_i,K_i,\epsilon_i,\widetilde{g\phi})$.  This proves that the map described in the statement is continuous, as required.
\end{proof}

\section{The Morse Inequalities and Other Results.}\label{ApplicationsSection}

This section contains three applications of the material on Morse functions developed so far, concluding with a proof of the Morse Inequalities.  In \S\ref{RepresentabilitySection} we give a criterion for a differentiable Deligne-Mumford stack $\X$ to be representable (i.e.~equivalent to a manifold) in terms of a Morse function $f\colon\X\to\R$.  In \S\ref{HandleSection} we will describe how the topology of the subset $\bar\X^a=\bar f^{-1}(-\infty,a]$ changes as one increases $a$, concluding with a description of the homotopy type of $\bar\X$.  Then in \S\ref{MorseInequalitiesSection} we state and prove the Morse Inequalities.

\subsection{Morse functions and representability.}\label{RepresentabilitySection}

Let $\X$ be a differentiable Deligne-Mumford stack.  Recall that we write $\bar\X^a$ for $\bar f^{-1}(-\infty,a]$.

\begin{theorem}
Let $f\colon\X\to\R$ be Morse with $\bar\X^a$ compact for each $a\in\R$.  Then $\X$ is representable if and only if each critical point of $f$ has trivial automorphism group.
\end{theorem}

By Corollary~\ref{ProperMorseFunctionCorollary} a Morse function $f$ with the required property always exists.  Note that the additional condition on $f$ is necessary: if $\Y$ is any differentiable Deligne-Mumford stack then the projection $\Y\times\R\to\R$ is Morse and has no critical points, regardless of whether $\Y\times\R$ is representable.  The theorem is a direct consequence of the next two propositions.  The first gives a simple criterion for $\X$ to be equivalent to a manifold.  The second states that the automorphism group of a general point in $\X$ is dominated by the inertia groups of the critical points of a Morse function $f$.  

\begin{proposition}\label{RepresentableProposition}
$\X$ is representable if and only if each of its points has trivial automorphism group.
\end{proposition}

\begin{note}
Proposition~\ref{RepresentableProposition} fails for more general stacks, as one sees by considering the stack $\mathfrak{A}$ of Examples~\ref{ExampleTwo}, \ref{ExampleThree}.  
\end{note}

\begin{proposition}\label{MorseAutomorphismProposition}
Let $f\colon\X\to\R$ be Morse with $\bar\X^a$ compact for all $a\in\R$.  Then for each $x\in\bar\X$ there is a critical point $c$ of $f$ and an injection $\aut_x\hookrightarrow\aut_c$.
\end{proposition}

\begin{proof}[Proof of Proposition~\ref{RepresentableProposition}.]
Using orbifold-charts and the fact that each point has trivial inertia we see that any point $x\in\bar\X$ is represented by some point in an open embedding $U\to\X$ from a manifold $U$; we may assume that $U\subset\R^n$ if we wish.

Let $U\to\X$ be an open embedding, $V\to\X$ any morphism.  By applying the underlying space functor to the cartesian diagram
\[\xymatrix{
U\times_\X V\ar[r]\ar[d]& V\ar[d]_{}="1"\\
U\ar[r]^{}="2"& \X\ar@{=>}"1";"2"
}\]
we obtain a diagram
\[\xymatrix{
U\times_\X V\ar[r]\ar[d]& V\ar[d]_{}="1"\\
U\ar[r]^{}="2"& \bar\X
}\]
in which the horizontal maps are open embeddings (by Proposition~\ref{UnderlyingMapProposition}), and which is still cartesian since points of $\X$ have no nontrivial automorphisms.

We may cover the second-countable Hausdorff space $\bar\X$ with the open inclusions $U\hookrightarrow\bar\X$ underlying open inclusions $U\to\X$ from open subsets of $\R^n$.  The previous paragraph shows first that these give a smooth atlas on $\bar\X$, so that $\bar\X$ is a manifold, and second that if $V\to\X$ is a morphism then the underlying $V\to\bar\X$ is smooth.  Thus there is a morphism
\begin{gather*}
\X\to\bar\X\\
(V\to\X)\mapsto(V\to\bar\X).
\end{gather*}
Finally the last paragraph shows that if $A\to\X$ is an \'etale atlas then so is $A\to\bar\X$; it also shows that the atlases induce the same groupoid $A\times_\X A\rightrightarrows A$ and that $\X\to\bar\X$ is covered by the identity map on this groupoid.  Thus $\X\to\bar\X$ is an equivalence and the result follows.
\end{proof}

\begin{proof}[Proof of Proposition~\ref{MorseAutomorphismProposition}.]
Let $x\in\bar\X$ with $\bar f(x)=a$.  Let $\lambda\colon\X\to\R$ be a compactly supported morphism which is equal to $1$ on a neighbourhood of $\bar f^{-1}(-\infty,a]$.  This exists by Lemma~\ref{BumpFunctionsLemma}.  Let $\Phi\colon\X\times\R\to\X$ be a flow of the compactly-supported vector field $\lambda\nabla f$ and let $\varphi_t\colon\bar\X\to\bar\X$ be the $\R$-action underlying $\Phi$.

The proof of \cite[Lemma 2.2.3]{\Invitation} can be adapted, using $\bar\X$ in place of $M$, using Lemma~\ref{UnderlyingFlowLemma}, and using the fact that $\bar\X$ is metrizable (by Proposition~\ref{DMUnderlyingSpaceProposition}), to show that
\[\lim_{t\to-\infty}\varphi_t(x)\]
exists and is a critical point $c$ of $f$.  Let $[U_c/\aut_c]\to\X$ be an orbifold chart at $c$.  Since $\varphi_t(x)\to c$ as $t\to-\infty$ we may choose $t$ with $|t|$ large enough that $\varphi_t(x)$ is the image under $U_c\to\bar\X$ of some $u\in U_c$.  Thus $\aut_{\varphi_t(x)}\cong\mathrm{Stab}_{\aut_c}(u)\subset \aut_c$.  However $\varphi_t$ is the map underlying the self-equivalence $\Phi|_{\X\times\{t\}}$ of $\X$ and so there is an isomorphism $\aut_x\cong\aut_{\varphi_t(x)}$.  The result follows.
\end{proof}

\subsection{The topology of the underlying space.}\label{HandleSection}

\begin{theorem}[c.f.~{\cite[3.1]{\Milnor}}]\label{RetractTheorem}
Let $a<b$ and suppose that $\bar f^{-1}[a,b]$ is compact and contains no critical points of $f$.  Then $\bar\X^a$ and $\bar\X^b$ are homeomorphic.  Moreover, $\bar\X^a$ is a strong deformation retract of $\bar\X^b$.
\end{theorem}

\begin{theorem}[c.f.~{\cite[3.2]{\Milnor}}]\label{HandleTheorem}
Let $p\in\bar\X$ be a nondegenerate critical point of $f$, let $c=\bar f(p)$ be the corresponding critical value, and suppose that there is $\epsilon>0$ such that $\bar f^{-1}[c-\epsilon,c+\epsilon]$ is compact and contains no critical points of $f$ besides $c$.  Then $\bar\X^{c+\epsilon}$ has the homotopy type of $\bar\X^{c-\epsilon}$ with a copy of $\mathbb{D}(\ind_p)/\aut_p$ attached along $\mathbb{S}(\ind_p)/\aut_p$.
\end{theorem}

Here $\mathbb{D}(-)$ and $\mathbb{S}(-)$ denote the unit disc and unit sphere in a representation, equipped with a suitable metric.  The analogue of Theorem~\ref{HandleTheorem} when there are several critical points with critical value $c$ holds with the obvious changes.

Now by Corollary~\ref{ProperMorseFunctionCorollary} we may assume that $f$ is Morse and that each $\bar\X^a$ is compact.  By Corollary~\ref{CriticalPointsIsolatedCorollary} the critical points of $f$ are isolated and so we may take $b_0<b_1<b_2<\cdots\in\R$ such that each interval $(b_i,b_{i+1})$ contains precisely one critical value of $f$ and such that all critical values lie in such an interval.  Write the critical points in $f^{-1}(b_i,b_{i+1})$ as $c^i_1,\ldots,c^i_{r_i}$.  We then have:

\begin{corollary}\label{HomotopyTypeCorollary}
$\bar\X=\bigcup_{i=0}^\infty\bar\X^{b_i}$, and each $\bar\X^{b_{i+1}}$ has the homotopy type of $\bar\X^{b_{i}}$ with copies of $\mathbb{D}(\ind_{c^i_{j}})$ attached along $\mathbb{S}(\ind_{c^i_{j}})$, for $j=1,\ldots,r_i$.
\end{corollary}

\begin{note}
It would be desirable to prove theorems describing the `topology of $\X$' rather than the topology of $\bar\X$.  For example, one might hope for a theorem describing $\X$ as a `handlebody' obtained by attaching handles $[\mathbb{D}(\ind_c)\times\mathbb{D}(\coind_c)/\aut_x]$, one for each critical point $c$ of $f$.  (Here we have written $\mathbb{D}(-)$ to denote the closed unit disc in a representation.)  Such a result would require a notion of differentiable Deligne-Mumford stack \emph{with boundary}, which is beyond the scope of the present paper.  In \cite{\OrbifoldMSW} we will return to this issue after developing a theory of differentiable Deligne-Mumford stacks with corners.
\end{note}

\begin{proof}[Proof of Theorem~\ref{RetractTheorem}]
This proof will follow that of Theorem~3.1 in \cite{\Milnor}, which we will refer to throughout.

Using Theorem~\ref{MetricTheorem}, let $\X$ be equipped with a Riemannian metric, so that we may form the vector field $\nabla f$ and the associated function $\|\nabla f\|^2\colon\X\to\R$ using Definition~\ref{ManipulateVectorFieldsDefinition}.  Using Lemma~\ref{ExtensionByZeroExistsLemma} and Corollary~\ref{BumpFunctionsLemma} we may find $\rho\colon\X\to\R$ with compact support and with $\bar\rho=1/\overline{\|\nabla f\|^2}$ on $\bar f^{-1}[a,b]$.  Finally, we may use Proposition~\ref{VectorFieldProposition} to form the vector field $X=\rho\nabla f$ on $\X$, which has compact support, and then take a flow $\Phi$ of $X$ using Theorem~\ref{FlowTheorem}.

Now we may define
\[\varphi_t\colon\bar\X\to\bar\X\]
to be the 1-parameter family of automorphisms underlying $\Phi$.  For fixed $x\in\bar\X$ the function $t\mapsto\bar f(\varphi_t(x))$ is smooth by Lemma~\ref{UnderlyingFlowLemma} and if $\varphi_t(x)\in\bar f^{-1}[a,b]$ its derivative is $\overline{X\cdot f}(\varphi_t(x))=\overline{\langle\rho\nabla f,\nabla f\rangle}(\varphi_t(x))=1$.  The remainder of the proof of \cite[3.1]{\Milnor} now goes through without change to establish the result.
\end{proof}

\begin{proof}[Proof of Theorem~\ref{HandleTheorem}.]
This proof will follow that of Theorem~3.2 in \cite{\Milnor}, which we will refer to throughout.

Take an $\aut_p$-equivariant splitting $T_p\X=T_p\X_+\oplus T_p\X_-$ such that $H_{f,p}|_{T_p\X_{+}}$ is positive definite and $H_{f,p}|_{T_p\X_{-}}$ is negative definite.  Equip $T_p\X$ with the metric $H_{f,p}|_{T_p\X_{+}}\oplus -H_{f,p}|_{T_p\X_{-}}$.  Write elements $u\in T_p\X$ as $(u_+,u_-)$ where $u_+\in T_p\X_+$ and $u_-\in T_p\X_-$.

Now take an orbifold-chart $[U_p/\aut_p]\to\X$ at $p$ as in the Morse Lemma Theorem~\ref{MorseLemma}.  By reducing $\epsilon$ if necessary we may assume that $U_p\subset T_p\X$ contains the closed unit ball of radius $2\epsilon$.  

Let $\mu\colon\R\to\R$ be a smooth function satisfying the conditions
\begin{gather*}
\mu(0)>\epsilon,\\
\mu(r)=0\mathrm{\ for\ all\ }r\geqslant 2\epsilon,\\
-1<\mu'(r)\leqslant 0\mathrm{\ for\ all\ }r.
\end{gather*}
Now consider the morphism $[U_p/\aut_p]\to\R$ that when composed with $U_p\to[U_p/\aut_p]$ becomes $u\mapsto \mu(2\|u_-\|^2+\|u_+\|^2)$.  This morphism has compact support, so we may use Lemma~\ref{ExtensionByZeroExistsLemma} to extend it to a morphism $\X\to\R$, and then subtract it from $f$ to obtain $F\colon\X\to\R$.  

The following three assertions are directly analogous to those that appear in the proof of \cite[3.2]{\Milnor} and are proved in exactly the same way.
\begin{description}
\item[Assertion 1.]  $\bar F^{-1}(-\infty,c+\epsilon]$ coincides with $\bar\X^{c+\epsilon}$.
\item[Assertion 2.] The critical points of $F$ are precisely those of $f$.
\item[Assertion 3.] $\bar F^{-1}(-\infty,c-\epsilon]$ is a deformation retract of $\bar\X^{c+\epsilon}$.
\end{description}

Write $\bar F^{-1}(-\infty,c-\epsilon]$ as $\bar\X^{c-\epsilon}\cup H$, where $H$ is the closure of $\bar F^{-1}(-\infty,c-\epsilon]\setminus\bar\X^{c-\epsilon}$.  Consider the `cell' $e_p\subset\overline{[U_p/\aut_p]}\subset\bar\X$ given by the image of those $u\in U_p$ with $\|u_-\|\leqslant\epsilon$ and $u_+=0$.  The intersection $\partial e_p$ of $e_p$ with $\bar\X^{c-\epsilon}$ is the image of those $u\in U_p$ with $\|u_-\|=\epsilon$ and $u_+=0$.  Thus the pair $(e_p,\partial e_p)$ serves as the pair $(\mathbb{D}(\ind_p)/\aut_p,\mathbb{S}(\ind_p)/\aut_p)$ appearing in the statement.  The following assertion will therefore complete the proof of the theorem.

\begin{description}
\item[Assertion 4.] $\bar\X^{c-\epsilon}\cup e_p$ is a deformation retract of $\bar\X^{c-\epsilon}\cup H$.
\end{description}

Let $\tilde F\colon U_p\to\R$ be the composition of $F$ with $U_p\to\X$.  The proof will follow if we define an $\aut_p$-equivariant deformation retraction of $\tilde F^{-1}(-\infty,c-\epsilon]$ onto the union of $\{u\mid \|u_+\|^2-\|u_-\|^2\leqslant c-\epsilon\}$ with $\{u\mid u_+=0,\,\|u_-\|^2\leqslant\epsilon\}$.  The deformation retraction used to prove Assertion 4 in the proof of \cite[3.2]{\Milnor} can be directly translated to the current situation, and is immediately seen to be equivariant.  This completes the proof.
\end{proof}

\subsection{The Morse Inequalities.}\label{MorseInequalitiesSection}

In this section we will state and prove the Morse inequalities for orbifolds, generalizing the Morse inequalities for manifolds.  To do so we must choose an appropriate extension to orbifolds of the notion of homology or cohomology of a manifold.  We concentrate on the following three possibilities and for simplicity we take coefficients in $\C$.
\begin{enumerate}
\item $H_\ast(\X)$ and $H^\ast(\X)$, the \emph{homology} and \emph{cohomology} of $\X$, are simply the homology and cohomology of the underlying space $\bar\X$.

\item The \emph{inertia string topology} of $\X$ is the homology $H_\ast(\Lambda\X)$ of the inertia stack, equipped with an associative graded commutative `string product' that is directly analogous to the Chas-Sullivan product on the homology of the loopspace of a manifold.  This first appeared in Lupercio et al.~\cite{\LUXLoop} as the `virtual cohomology of $\X$' and has been described in a different way by Behrend et al.~\cite{\GinotEtAl}.

\item The \emph{orbifold} or \emph{Chen-Ruan} cohomology $H^\ast_\orb(\X)$ of an almost-complex $\X$ is given by $H^\ast(\Lambda\X)$ equipped with a modified grading and an ingenious cup-product structure \cite{\ChenRuan}.  Ruan's crepant resolution conjecture relates $H^\ast_\orb(\X)$ to the cohomology ring $H^\ast(Y)$ of a crepant resolution $Y\to\X$ \cite{\RuanCrepant}.
\end{enumerate}

Now let $\X$ be a differentiable Deligne-Mumford stack with $\bar\X$ compact and let $f\colon\X\to\R$ be a Morse function.

\begin{definition}
We write
\begin{gather*}
b_i=\dim H_i(\bar\X;\R),\\
b^\Lambda_i=\dim H_i(\overline{\Lambda\X};\R),\\
b^\orb_i=\dim H^i_\orb(\X)
\end{gather*}
for the \emph{Betti numbers}, \emph{inertia Betti numbers} and \emph{orbifold Betti numbers} of $\X$ respectively and we write
\begin{gather*}
P_t(\X)=\sum b_it^i,\\
P^\Lambda_t(\X)=\sum b_i^\Lambda t^i,\\
P^\orb_t(\X)=\sum b_i^\orb t^i
\end{gather*}
for the corresponding \emph{Poincar\'e polynomials}.  (The orbifold Betti numbers and corresponding Poincar\'e polynomial are only defined when $\X$ is almost-complex.)
\end{definition}

\begin{definition}
The \emph{Morse polynomial} of $f$, the \emph{inertia Morse polynomial} of $f$, and the \emph{orbifold Morse polynomial} of $f$ are
\begin{gather*}
M_t(f)=\sum_c t^{\dim\ind_c},\\
M_t^\Lambda(f)=\sum_{c,(g)}t^{\dim{\ind_c}^g},\\
M_t^\orb(f)=\sum_{c,(g)}t^{\dim{\ind_c}^g+2\iota(c,g)},
\end{gather*}
respectively, where in the first case the sum is taken over \emph{oriented} critical points of $f$ and in the second and third cases the sum is taken over \emph{oriented} pairs $(c,(g))$ with $c$ a critical point of $f$ and $(g)$ a conjugacy class in $\aut_c$.  The third definition only applies when $\X$ is almost-complex, and in this case $\iota(c,g)$ denotes the \emph{age grading} or \emph{degree-shifting number} \cite{\ChenRuan}.
\end{definition}

\begin{theorem}[Orbifold Morse Inequalities.]\label{MorseInequalitiesTheorem}
There are polynomials $R(t)$, $R^\Lambda(t)$, $R^\orb(t)$ with non-negative integer coefficients such that
\begin{gather*}
M_t(f)=P_t(\X)+(1+t)R(t),\\
M_t^\Lambda(f)=P_t^\Lambda(\X)+(1+t)R^\Lambda(t),\\
M_t^\orb(f)=P_t^\orb(\X)+(1+t)R^\orb(t).
\end{gather*}
In particular, if $M_t(f)$ has no consecutive powers of $t$ then $M_t(f)=P_t(\X)$, and similarly for $M_t^\Lambda(f)$ and $M_t^\orb(f)$.  The third result only applies when $\X$ is almost-complex.
\end{theorem}

\begin{proof}
We will only prove the first part.  The second and third follow immediately from the first and Theorem~\ref{InertiaTheorem}.  See \cite{\ChenRuan} for a definition of $H^\ast_\orb(\X)$ and an explanation of degree-shifting.

To prove the first part we will make use of \cite[\S 5]{\Milnor}.  The function that assigns to a pair of spaces $(X,Y)$ the alternating sum
\[S_i(X,Y)=\dim H_i(X,Y;\R)-\dim H_{i-1}(X,Y;\R)+\cdots\pm\dim H_0(X,Y;\R)\]
is \emph{subadditive} in the sense that given $X\supset Y\supset Z$ we have $S_i(X,Z)\leqslant S_i(X,Y)+S_i(Y,Z)$.  Now taking $b_0<\cdots<b_r$ as in Corollary~\ref{HomotopyTypeCorollary} (the sequence terminates since $\X$ is compact) we have
\begin{equation}\label{InequalityEquation}S_i(\bar\X,\emptyset)=S_i(\bar\X^{b_r},\bar\X^{b_0})\leqslant\sum S_i(\bar\X^{b_p},\bar\X^{b_{p-1}})\end{equation}
where the inequality is a simple consequence of subadditivity.  Now by Corollary~\ref{HomotopyTypeCorollary} we have
\begin{eqnarray*}
H_i(\bar\X^{b_{p}},\bar\X^{b_{p-1}};\R)
&=&H_i\left(\bar\X^{b_{p-1}}\cup\bigcup\mathbb{D}(\ind_{c^p_j})/\aut_{c^p_j},\bar\X^{b_{p-1}};\R\right)\\
&=&\bigoplus H_i(\mathbb{D}(\ind_{c^p_j})/\aut_{c^p_j},\mathbb{S}(\ind_{c^p_j})/\aut_{c^p_j};\R),
\end{eqnarray*}
and $H_i(\mathbb{D}(\ind_{c^p_j})/\aut_{c^p_j},\mathbb{S}(\ind_{c^p_j})/\aut_{c^p_j};\R)$ is equal to $\R$  if $c^p_j$ is orientable and $i=\dim\ind_{c^p_j}$,
and is equal to $0$ otherwise.  Therefore \eqref{InequalityEquation} becomes
\begin{equation}\label{SecondInequalityEquation}S_i(\bar\X,\emptyset)\leqslant C_i-C_{i-1}+\cdots\pm C_0,\end{equation}
where $C_j$ denotes the number of orientable critical points $c$ of $f$ with $\dim\ind_c=j$.  Finally \cite[3.43]{\BanyagaHurtubise} shows that \eqref{SecondInequalityEquation} is equivalent to the first claim in the theorem.  This completes the proof.
\end{proof}

\section{Examples.}\label{ExamplesSection}

This section contains three examples that apply the results of Section~\ref{ApplicationsSection}.  The first two demonstrate how in favourable circumstances Theorem~\ref{MorseInequalitiesTheorem} allows us to compute the homology or orbifold cohomology groups of an orbifold.  The third example demonstrates how the methods can be extended to compute the integer homology groups of the $K3$ surface.

\begin{example}[Weighted projective spaces.]
Bott's perfect Morse function on complex projective space (see \cite[3.7]{\BanyagaHurtubise}) generalizes to the weighted projective spaces as follows.  Let $q_0,\ldots,q_n\in\mathbb{N}$ and set
\[\mathfrak{P}^n_{q_0\cdots q_n}=[S^{2n+1}/T^1],\]
where $T^1$ acts on $S^{2n+1}\subset\mathbb{C}^{n+1}$ by $t\cdot(z_0,\ldots,z_n)=(t^{q_0}z_0,\ldots,t^{q_n}z_n)$.  Thus $\overline{\mathfrak{P}}^n_{q_0\cdots q_n}$ is the weighted projective space $\mathbb{CP}^n_{q_0\cdots q_n}$.  We define
\[f\colon\mathfrak{P}^n_{q_0\cdots q_n}\to\R\]
to be the function that when composed with $S^{2n+1}\to \mathfrak{P}^n_{q_0\cdots q_n}$ becomes $(z_0,\ldots,z_n)\mapsto |z_1|^2+2|z_2|^2+\cdots+n|z_n|^2$.

The computations of \cite[3.7]{\BanyagaHurtubise} immediately generalize to show that $f$ is Morse and has critical points
\[c_i=[(0,\ldots,1,\ldots,0)\to\mathfrak{P}^n_{q_0\cdots q_n}],\]
where the nonzero entry is in the $i$th place.  The critical point data for $f$ are given by
\begin{eqnarray*}
\aut_{c_i}&=&\Z_{q_i},\\
\ind_{c_i}&=&\rho^{q_0}\oplus\cdots\oplus\rho^{q_{i-1}},\\
\coind_{c_i}&=&\rho^{q_{i+1}}\oplus\cdots\oplus\rho^{q_{n}}.
\end{eqnarray*}
Here $\rho$ denotes the linear complex representation of $\Z_{q_i}$ in which the generator acts by $z\mapsto e^{{2\pi i}/{q_i}}z$.  We immediately obtain from Theorem~\ref{MorseInequalitiesTheorem} the familiar computation of the rational homology of $\mathbb{CP}^n_{q_0\cdots q_n}$:
\[
H_{i}(\mathbb{CP}^n_{q_0\cdots q_n};\mathbb{Q})=\left\{\begin{array}{cl}\mathbb{Q}&\mathrm{if\ }i=2j\mathrm{\ for\ }0\leqslant j\leqslant n,\\
0&\mathrm{otherwise}.
\end{array}\right.
\]
\qed\end{example}

Our next two examples will consider the \emph{K\"ummer construction} of the \emph{$K3$ surface} \cite[6.6.1]{\JoyceExceptional}.  Let $\mathfrak{K}$ denote the quotient stack $[T^4/\Z_2]$, where the generator of $\Z_2$ acts on $T^4=\R^4/\Z^4$ by $(x_1,x_2,x_3,x_4)\mapsto(-x_1,-x_2,-x_3,-x_4)$.  The underlying space $\bar{\mathfrak{K}}=T^4/\Z_2$ is not a manifold since it contains 16 points with neighbourhoods of the form $\C^2/\Z_2$, with $\Z_2$ acting on $\C^2$ by negation of vectors.  However, there is a crepant resolution
\[\rho\colon T^\ast\mathbb{CP}^1\to\C^2/\Z_2\]
that sends $\mathbb{CP}^1$ to $\{0\}$ and that becomes a diffeomorphism after deleting these subsets.  We may therefore replace the 16 regions $\C^2/\Z_2$ in $\bar{\mathfrak{K}}$ with copies of $T^\ast\mathbb{CP}^1$ to obtain a manifold $K$ which is a $K3$ surface, and the crepant resolution
\[\pi\colon K\to\bar{\mathfrak{K}}.\]

It is possible to show that $\pi_1(K)=0$ and to compute $H_\ast(K;\mathbb{Q})$ directly.  One can then infer, from the fact that $K$ is a simply-connected $4$-manifold, that
\[H_i(K;\Z)=\left\{\begin{array}{cl}\Z &\mathrm{if\ }i=0\mathrm{\ or\ }4,\\ \Z^{22}&\mathrm{if\ }i=2,\\  0 &\mathrm{otherwise.} \end{array}\right.\]

\begin{question}
Is it possible to compute $H_\ast(K;\Z)$ in a way that will generalize to other manifolds constructed as resolutions of orbifolds?
\end{question}

The manifolds that we have in mind when asking this question include crepant resolutions of Calabi-Yau orbifolds (\cite{\JoyceDesingularizations} for example) and Joyce's compact manifolds with $G_2$ and $\mathrm{Spin}(7)$ holonomy \cite{\JoyceExceptional}.

\begin{example}[The K\"ummer construction and orbifold cohomology.]\label{KummerOrbifoldExample}
Chen and Ruan \cite[5.1]{\ChenRuan} have computed $H^\ast_\orb(\mathfrak{K})$:
\begin{equation}\label{KThreeEquation}
H^i_\orb(\mathfrak{K})=\left\{\begin{array}{cl}\mathbb{C} &\mathrm{if\ }i=0\mathrm{\ or\ }4,\\ \mathbb{C}^{22}&\mathrm{if\ }i=2,\\  0 &\mathrm{otherwise} \end{array}\right.
\end{equation}
so that as a vector-space $H^\ast_\orb(\mathfrak{K})$ is isomorphic to $H^\ast(K;\C)$.  We will recover this result using Morse theory.

Let $f\colon\mathfrak{K}\to\R$ be the morphism that when composed with $T^4\to\mathfrak{K}$ becomes $(x_1,x_2,x_3,x_4)\mapsto\sum\cos(2\pi x_i)$.  This is Morse and has 16 critical points
\[c_{ijkl}=\left[\left({\textstyle\frac{i}{2}},{\textstyle\frac{j}{2}},{\textstyle\frac{k}{2}},{\textstyle\frac{l}{2}}\right)\to\mathfrak{K}\right],\quad i,j,k,l\in\{0,1\}.\]
The critical point data are
\begin{eqnarray*}
\aut_{c_{ijkl}}&=&\Z_2,\\
\ind_{c_{ijkl}}&=&{(4-i-j-k-l)}(-\mathbf{1}),\\
\coind_{c_{ijkl}}&=&({i+j+k+l})(-\mathbf{1}),
\end{eqnarray*}
where $-\mathbf{1}$ is the nontrivial linear representation of $\Z_2$.  It follows that $c_{ijkl}$ is orientable if and only if $i+j+k+l$ is even.  Further, $\iota(c_{ijkl}^0)=0$ and $\iota(c_{ijkl}^{1})=1$.  We therefore have
\begin{eqnarray*}M^\orb_t(f)&=&\sum_{i+j+k+l\ \mathrm{even}}t^{4-i-j-k-l}+\sum_{i,j,k,l}t^2\\
&=&1+22t^2+t^4.\end{eqnarray*}
Since this is concentrated in even degrees we immediately recover \eqref{KThreeEquation} from Theorem~\ref{MorseInequalitiesTheorem}.
\qed\end{example}

\begin{note}\label{ComputationNote}
Let $V\subset\mathbb{C}^2$ be any \emph{real} linear subspace and consider again the projection $\rho\colon T^\ast\mathbb{CP}^1\to\mathbb{C}^2/\Z_2$.  It is not difficult to show that the total homology group
\[H_\ast(\rho^{-1}(\mathbb{D}V/\Z_2),\rho^{-1}(\mathbb{S}V/\Z_2);\Z)\]
is \emph{free} on one generator in degree two and --- if $\dim V$ is even --- an additional generator in degree $\dim V$.  Note that the additional generator is introduced if and only if the $\Z_2$-action on $V$ preserves orientations.

Rationally this computation is entirely trivial but proving that the groups are free involves an extension problem.  We regard the result as the best possible, bearing in mind the crepant resolution conjecture.
\end{note}

\begin{example}[The integer homology of the $K3$ surface.]\label{KummerIntegralExample}
We will now show how to use the Morse function $f\colon\mathfrak{K}\to\R$ of Example~\ref{KummerOrbifoldExample}, together with the resolution $\pi\colon K\to\bar{\mathfrak{K}}$, to compute the integer homology of $K$.  The idea is that we will treat the composition $\bar f\circ\pi\colon K\to\R$ as if it were a Morse function and use it to understand the homology of $K$ by considering the sublevel sets $K^a=(\bar f\circ\pi)^{-1}(-\infty,a]$.  Note that $\bar f\circ\pi$ is not a Morse function, and indeed is not smooth.  It nonetheless provides an excellent tool for studying $K$.

By adapting the methods used to prove Theorems \ref{RetractTheorem} and \ref{HandleTheorem} it is possible to show that:
\begin{enumerate}
\item $K^a\cong K^b$ if $f$ has no critical values between $a$ and $b$.
\item Let $c$ be a critical value of $f$.  Then for $\epsilon>0$ sufficiently small $K^{c+\epsilon}$ has the homotopy type of $K^{c-\epsilon}$ with a copy of $\rho^{-1}(\mathbb{D}(\ind_p)/\Z_2)$ attached along $\rho^{-1}(\mathbb{S}(\ind_p)/\Z_2)$ for each critical point of $f$ with value $c$.
\end{enumerate}
Set $b_i=i-\frac{1}{2}$ for $i=0,\ldots,5$, so that each interval $[b_{i},b_{i+1}]$ contains the single critical value $i$ of $f$.  The two facts above, together with the computations of Note~\ref{ComputationNote} and the description of the critical point data of $f$ from Example~\ref{KummerOrbifoldExample} allow us to immediately compute that the homology groups
\[H_\ast(K^{b_{i}},K^{b_{i-1}};\Z)\]
are free on $\binom{4}{i}$ generators of degree $2$ with an additional $\binom{4}{i}$ generators in degree $i$ if $i$ is even.  Using the exact sequences
\[\cdots\to H_i(K^{b_{i-1}};\Z)\to H_i(K^{b_i};\Z)\to H_i(K^{b_i},K^{b_{i-1}};\Z)\to\cdots,\]
we may immediately conclude that
\[H_i(K;\Z)=\left\{\begin{array}{cl}\Z &\mathrm{if\ }i=0\mathrm{\ or\ }4,\\ \Z^{22}&\mathrm{if\ }i=2,\\  0 &\mathrm{otherwise}. \end{array}\right.\]
\qed\end{example}

\begin{note}
A full account of the methods used in Example~\ref{KummerIntegralExample} will be given in \cite{\Crepant}.  The methods can be extended to compute, for example, the homology of the crepant resolutions appearing in \cite{\JoyceDesingularizations} and the homology of Joyce's `simple example' of a compact manifold with $G_2$-holonomy \cite[12.2]{\JoyceExceptional}.
\end{note}

\bibliographystyle{alpha}
\bibliography{MorseTheoryBibliography}
\end{document}